\documentclass[12pt, reqno]{amsart}

\usepackage{amssymb, amsmath, amsthm}
\usepackage[backref]{hyperref}
\usepackage[alphabetic,backrefs,lite]{amsrefs}
\usepackage{amscd}   
\usepackage{fullpage}
\usepackage[all]{xy}

\DeclareFontEncoding{OT2}{}{} 

\usepackage[usenames,dvipsnames]{color}

\newtheorem{lemma}{Lemma}[section]
\newtheorem{theorem}[lemma]{Theorem}
\newtheorem{proposition}[lemma]{Proposition}
\newtheorem{prop}[lemma]{Proposition}
\newtheorem{cor}[lemma]{Corollary}

\newtheorem{claim*}{Claim}
\newtheorem{thm}[lemma]{Theorem}
\newtheorem{defn}[lemma]{Definition}

\theoremstyle{definition}
\newtheorem{remark}[lemma]{Remark}

\newcommand{\PP}{{\mathbb P}}
\newcommand{\C}{{\mathbb C}}
\newcommand{\F}{{\mathbb F}}
\newcommand{\Q}{{\mathbb Q}}

\newcommand{\Z}{{\mathbb Z}}

\newcommand{\Xbar}{{\overline{X}}}

\newcommand{\kbar}{{\overline{k}}}

\newcommand{\Ybar}{{\overline{Y}}}
\newcommand{\Abar}{{\overline{A}}}
\newcommand{\Ebar}{{\overline{E}}}
\newcommand{\Epbar}{{\overline{E'}}}
\newcommand{\EEpbar}{{\overline{E\times E'}}}

\newcommand{\LamBBp}{\Lambda_{B,B'}}

\newcommand{\dB}{\widehat{B}}

\newcommand{\calO}{{\mathcal O}}

\newcommand{\OO}{{\mathcal O}}

\newcommand{\frakp}{{\mathfrak p}}

\DeclareMathOperator{\HH}{H}

\DeclareMathOperator{\Tr}{Tr}

\DeclareMathOperator{\coker}{coker}
\DeclareMathOperator{\rk}{rk}

\DeclareMathOperator{\im}{im}

\DeclareMathOperator{\End}{End}

\DeclareMathOperator{\Hom}{Hom}

\DeclareMathOperator{\Gal}{Gal}

\DeclareMathOperator{\Br}{Br}

\DeclareMathOperator{\ord}{ord}
\DeclareMathOperator{\Sym}{Sym}

\DeclareMathOperator{\Pic}{Pic}
\DeclareMathOperator{\Jac}{Jac}

\DeclareMathOperator{\Spec}{Spec}

\DeclareMathOperator{\Tor}{Tor}

\DeclareMathOperator{\et}{et}

\DeclareMathOperator{\Disc}{Disc}

\DeclareMathOperator{\GL}{GL}

\DeclareMathOperator{\id}{id}

\DeclareMathOperator{\vol}{vol}

\DeclareMathOperator{\NS}{NS}
\DeclareMathOperator{\Kum}{Kum}
\DeclareMathOperator{\Nm}{Nm}

\renewcommand{\et}{\textnormal{\'et}}
\newcommand{\largewedge}{\mbox{\Large $\wedge$}}

\newcommand{\ol}{\overline}

\numberwithin{equation}{section}
\numberwithin{table}{section}

\newcommand{\defi}[1]{\textsf{#1}} 
\newcommand{\inter}[2]{(#1.#2)} 

\title{Effective bounds for Brauer groups of Kummer surfaces over number fields}
\author{Victoria Cantoral Farf\'{a}n}
\author{Yunqing Tang}
\author{Sho Tanimoto}
\author{Erik Visse}

\address{Institut de Math\'ematiques de Jussieu - Paris Rive Gauche (IMJ-PRG)
UP7D - B\^atiment Sophie Germain - 75205 Paris France}
\email{victoria.cantoral-farfan@imj-prg.fr}
\urladdr{http://webusers.imj-prg.fr/~victoria.cantoral-farfan/}

\address{Mathematics Department, Harvard University, 1 Oxford Street, Cambridge, MA 02138, USA}
\email{yqtang@math.harvard.edu}
\urladdr{http://www.math.harvard.edu/~yqtang/}

\address{Department of Mathematical Sciences, University of Copenhagen, Universitetspark 5 2100 Copenhagen $\emptyset$ Denmark}
\email{sho@math.ku.dk}
\urladdr{http://shotanimoto.wordpress.com}

\address{Mathematisch Instituut, Leiden University, Niels Bohrweg 1, 2333CA, Leiden, the Netherlands}
\email{h.d.visse@math.leidenuniv.nl}
\urladdr{http://pub.math.leidenuniv.nl/$\sim$vissehd/}

\date{}

\begin{document}

	\begin{abstract}
		We study effective bounds for Brauer groups of Kummer surfaces associated to Jacobians of genus $2$ curves defined over number fields.
	\end{abstract}

	\maketitle
	

\section{Introduction}\label{Intro}

In 1971, Manin observed that failures of Hasse principle and weak approximation can be explained by Brauer-Manin obstructions for many examples \cite{Man71}.
Let $X$ be a smooth projective variety defined over a number field $k$.
\defi{The Brauer group of $X$} is defined as $$\Br(X) := \HH^2_{\et}(X, \mathbb G_m).$$
Then one can define an intermediate set using class field theory
\begin{displaymath}
X(k) \subset X(\mathbb A_k)^{\Br(X)} \subset X(\mathbb A_k),
\end{displaymath}
where $\mathbb A_k$ is the ad\`elic ring associated to $k$.
It is possible that $X(\mathbb A_k) \neq \emptyset$, but $X(\mathbb A_k)^{\Br(X)} = \emptyset$, whereby the Hasse principle fails for $X$.
When this happens, we say that there is \defi{a Brauer-Manin obstruction to the Hasse principle}.
When $X(\mathbb A_k)^{\Br(X)} \neq X(\mathbb A_k)$, we say that there is \defi{a Brauer-Manin obstruction to weak approximation}.
There is a large body of work on Brauer-Manin obstructions to the Hasse principle and weak approximation
(see, e.g., \cite{Man74}, \cite{BS75}, \cite{CTCJ80}, \cite{CTSSD87}, \cite{CTKS87}, \cite{SD93}, \cite{SD99}, \cite{KT04}, \cite{Bri06}, \cite{BBFL07}, \cite{KT08}, \cite{Log08}, \cite{VA08}, \cite{LvL09}, \cite{EJ10}, \cite{HVAV11}, \cite{ISZ},  \cite{EJ12}, \cite{HVA13}, \cite{CS13}, \cite{MSTVA14}, \cite{SZ14}, \cite{IS15}, \cite{OW16})
and it is an open question if for K3 surfaces, Brauer-Manin obstructions suffice to explain failures of Hasse principle and weak approximation, i.e., $X(k)$ is dense in $X(\mathbb A_k)^{\Br(X)}$
(see \cite{HS15} for some evidence supporting this conjecture.)

The main question discussed in this paper is of computational nature: how can one compute $\Br(X)$ explicitly? It is shown by Skorobogatov and Zarhin in \cite{SZ08} that $\Br(X)/\Br(k)$ is finite for any K3 surface $X$ defined over a number field $k$, but they did not provide any effective bound for this group.
Such an effective algorithm is obtained for degree $2$ K3 surfaces in \cite{HKT13} using explicit constructions of moduli spaces of degree $2$ K3 surfaces and principally polarized abelian varieties.
In this paper, we provide an effective algorithm to compute a bound for $\Br(X)/\Br(k)$ when $X$ is the Kummer surface associated to the Jacobian of a curve of genus $2$:

\begin{thm}
\label{thm:main}
There is an effective algorithm that takes as input an equation of a smooth projective curve $C$ of genus $2$ defined over a number field $k$, and outputs an effective bound for $\Br(X)/\Br_0(X)$ where $X$ is the Kummer surface associated to the Jacobian $\Jac(C)$ of the curve $C$.
\end{thm}
We obtain the following corollary as a consequence of results in \cite{KT11} and \cite{PTV}:

\begin{cor}
Given a smooth projective curve $C$ of genus $2$ defined over a number field $k$, 
there is an effective description of the set
\begin{displaymath}
 X(\mathbb A_k)^{\Br(X)}
\end{displaymath}
where $X$ is the Kummer surface associated to the Jacobian $\Jac(C)$ of the curve $C$.
\end{cor}

Note that given a curve $C$ of genus $2$, the surface $Y=\Jac(C)/\{\pm 1\}$ can be realized as a quartic surface in $\mathbb P^3$ (see \cite{FS97} Section 2) and the Kummer surface $X$ associated to $\Jac(C)$ is the minimal resolution of $Y$, so one can find defining equations for $X$ explicitly.

The quartic surface $Y$ has sixteen nodes, and by considering the projection from one of these nodes, we may realize $Y$ as a double cover of the plane. Thus $X$ can be realized as a degree $2$ K3 surface and our Theorem~\ref{thm:main} follows from \cite{HKT13}. However we avoid the use of the Kuga--Satake construction which makes our algorithm more practical than the method in \cite{HKT13}. In particular, our algorithm provides a large, but explicit bound for the Brauer group of $X$. (See the example we discuss below.)

The method in this paper combines many results from the literature.
 The first key observation is that the Brauer group $\Br(X)$ admits the following stratification:

\begin{defn}
Let $\ol X$ denote $X\times_k \Spec{\ol k}$ where $\ol k$ is a given separable closure of $k$. Then we write $\Br_0(X)=\im\left(\Br(k)\to \Br(X)\right)$ and $\Br_1(X)=\ker\left(\Br(X)\to \Br(\ol X)\right)$.

Elements in $\Br_1(X)$ are called \defi{algebraic elements}; those in the complement  $\Br(X) \setminus \Br_1(X)$ are called \defi{transcendental elements}.
\end{defn}

Thus to obtain an effective bound for $\Br(X)/\Br_0(X)$, it suffices to study $\Br_1(X)/\Br_0(X)$ and $\Br(X)/\Br_1(X)$. The group $\Br_1(X)/\Br_0(X)$ is well-studied, and it admits the following isomorphism:
\begin{displaymath}
\Br_1(X)/\Br_0(X) \cong \HH^1(k, \Pic(\ol X)).
\end{displaymath}
Note that for a K3 surface $X$, we have an isomorphism $\Pic(X) = \NS(X)$.
Thus as soon as we compute $\NS(\ol X)$ as a Galois module, we are able to compute $\Br_1(X)/\Br_0(X)$. An algorithm to compute $\NS(\ol X)$ is obtained in \cite{PTV}, but we consider another algorithm which is based on \cite{FC}.

To study $\Br(X)/\Br_1(X)$, we use effective versions of Faltings' theorem and combine them with techniques in \cite{SZ08} and \cite{HKT13}. Namely, we have an injection
\begin{displaymath}
\Br(X)/\Br_1(X) \hookrightarrow \Br(\ol X)^\Gamma
\end{displaymath}
where $\Gamma$ is the absolute Galois group of $k$.
As a consequence of \cite{SZ12}, we have an isomorphism of Galois modules
\begin{displaymath}
\Br(\ol X) = \Br(\ol A),
\end{displaymath}
where $A=\Jac(C)$ is the Jacobian of $C$.
Thus it suffice to bound the size of $\Br(\ol A)^\Gamma$.
To bound the cardinal of this group, we consider the following exact sequence as \cite{SZ08}:
\begin{align*}
\begin{split}
0 &\to \left(\NS(\ol{A})/\ell^n\right)^\Gamma \stackrel{f_n}{\to} \HH^2_\et(\ol{A},\mu_{\ell^n})^\Gamma \to \Br(\ol{A})_{\ell^n}^\Gamma\to \\
&\to \HH^1(\Gamma, \NS(\ol{A})/\ell^n) \stackrel{g_n}{\to} \HH^1(\Gamma,\HH^2_\et(\ol{A},\mu_{\ell^n})),
\end{split}
\end{align*}
where $\ell$ is any prime and $\Br(\ol{A})_{\ell^n}$ is the $\ell^n$-torsion part of the Brauer group of $\ol A$. Using effective versions of Faltings' theorem, we bound the cokernel of $f_n$ and the kernel of $g_n$ independently of $n$.

We emphasize that our algorithm is practical for any genus $2$ curve whose Jacobian has N\'{e}ron--Severi rank $1$, i.e., we can actually implement and compute a bound for such a curve. For example, consider the following hyperelliptic curve of genus 2 defined over $\mathbb Q$:
\begin{displaymath}
C: y^2 = x^6+x^3+x+1.
\end{displaymath}
Let $A = \Jac(C)$ and let $X = \Kum(A)$ be the Kummer surface associated to $A$.
The geometric N\'{e}ron--Severi rank of $A$ is $1$.
Our algorithm shows that
\begin{displaymath}
|\Br(X)/\Br(\mathbb Q)| <  4\cdot 10^{7.5\cdot 10^{16106}}.
\end{displaymath}

Our effective bound explicitly depends on the Faltings height of the Jacobian of $C$, so it does not provide any uniform bound as conjectured in \cite{TVA15}, \cite{AVA16}, and \cite{VA16}.
However, it is an open question whether the Faltings height in Theorem~\ref{effFal} is needed.
If there is a uniform bound for Theorem~\ref{effFal} which does not depend on the Faltings height,
then our proof provides a uniform bound for the Brauer group.
Such a uniform bound is obtained for elliptic curves in \cite{VAV16}.

Some theory behind the computation is given in Section \ref{Alg_part1} and actual computations using \textsc{Magma} are described in Section \ref{computations}.

The paper is organized as follows. In Section~\ref{Effective_Faltings} we review effective versions of Faltings' theorem and consequences that will be useful for our purposes. In Section~\ref{Alg_part1} we review methods from the literature in order to compute the N\'{e}ron--Severi lattice as a Galois module. Section~\ref{Trans_part} proves our bounds for the size of the transcendental part. Section~\ref{computations} is devoted to \textsc{Magma} computations in the lowest rank case and Section~\ref{section:example} explores an example.

\noindent
{\bf Acknowledgments}

The authors would like to thank Martin Bright, Edgar Costa, Brendan Hassett, Hendrik Lenstra, Ronald van Luijk, Chloe Martindale, Rachel Newton, Fabien Pazuki, Dan Petersen, Padmavathi Srinivasan, and Yuri Tschinkel for useful discussions and comments. In particular we would like to thank Rachel Newton for her comments on the early draft of this paper. They also would like to thank Andreas-Stephan Elsenhans for providing us with the \textsc{Magma} code of the algorithm in \cite{EJ}. 

This paper began as a project in Arizona Winter School 2015 ``Arithmetic and Higher-dimensional varieties''.
The authors would like to thank AWS for their hospitality and travel support. Finally the authors are grateful to Tony V\'{a}rilly-Alvarado for suggesting this project, many conversations where he patiently answered our questions, and for his constant encouragement.
This project and AWS have been supported by NSF grant DMS-1161523.
Tanimoto is supported by Lars Hesselholt's Niels Bohr professorship.

\section{Effective version of Faltings' theorem}\label{Effective_Faltings}

One important input of our main theorem is an effective version of Faltings' isogeny theorem. Such a theorem was first proved by Masser and W\"ustholz in \cite{MW95} and the computation of the constants involved was made explicit by Bost \cite{B96} and Pazuki \cite{P12}. The work of Gaudron and Remond \cite{GR14} gives a sharper bound. Although the general results are valid for any abelian variety over a number field, we will only focus on elliptic curves and abelian surfaces. 

The bounds in this section depend on the stable Faltings height of the given abelian surface. If a hyperelliptic curve $C$ is given by $y^2+G(x)y=F(x)$, where $G(x), F(x)$ are polynomials in $x$ of degrees at most $3$ and $6$ respectively, then an upper bound for the height of $\Jac(C)$ can be computed using \cite{P14}*{Thm.~2.4}.
More precisely, the functions \textit{AnalyticJacobian} and \textit{Theta} in \textsc{Magma} compute the period matrix of $\Jac(C)$ and the theta functions used to define $J_{10}$ in Pazuki's formula. 

Let $k'$ be a finite extension of $k$ such that after base change to $k'$, the variety $\Jac(C)_{k'}$ has semistable reduction everywhere. For example, $k'$ can be taken to be the field of definition of all $12$-torsion points. 

To bound the non-archimedean contribution to Pazuki's formula \cite{P14}*{Thm.~2.4}, we notice that at each finite place $v$, the local contribution is bounded by the minimum of 
\begin{displaymath}
\frac{1}{10}\ord_v(2^{-12}\Disc_{6}(4F_v+G_v^2))\log N_{k'/\Q}(v),
\end{displaymath}
since $e_v$ defined by Pazuki is non-negative (see \cite{P14}*{Def.~8.2, Prop.~8.6}). Here $F_v(x)$ and $G_v(x)$ are polynomials of degrees at most $6$ and $3$ in $\OO_{k'_v}[x]$ such that $C_{k'_v}$ is defined by $y^2+G_v(x)y=F_v(x)$ and the minimum is taken over all such polynomials $F_v$ and $G_v$. 
Hence if $F(x), G(x)\in \OO_k[x]\subset \OO_{k'}[x]$, then we bound the sum of the contributions of all non-archimedean places by $\frac{1}{10}\log(2^{-12}\Disc_{6}(4F+G^2))$.

We also remark that following \cite{K99}*{Sec.~4,5} one can easily compute the exact local contribution at $v\nmid 2$ by studying the roots of $F(x)$ assuming $G=0$. 

Let $A$ be an abelian surface defined over a number field $k$. Let $\Gamma$ be its absolute Galois group. We denote the stable Faltings height of $A$ by $h(A)$ (with the normalization as in the original work of Faltings \cite{F}). For a positive integer $m$, let $A_m$ be the $\Z[\Gamma]$-module of $m$-torsion points of $A(\bar{k})$. Without further indication, $A$ will be the Jacobian of some hyperelliptic curve $C$, principally polarized by the theta divisor, and we use $L$ to denote the line bundle on $A$ corresponding to the theta divisor.

Throughout this section, when we say there is an isogeny between abelian varieties $A$ and $B$ of degree at most $D$, it means that there exist isogenies $A\rightarrow B$ and $B\rightarrow A$ both whose degrees are at most $D$.

\subsection{geometrically simple case}
\label{subsec: simple}
We first deal with the case when $A$ is geometrically simple. Equivalently, $A$ is not isogenous to a product of two elliptic curves over $\bar{k}$.

The following theorem is a combination of results in \cite{MW95} and \cite{GR14}. 

\begin{theorem}\label{effFal}
For any integer $m$, there exists a positive integer $M_m$ such that the cokernel of the map $\End_k(A)\rightarrow \End_\Gamma(A_m)$ is killed by $M_m$. Furthermore, there exists an upper bound for $M_m$ depending on $h(A)$ and $[k:\Q]$ which is independent of $m$. Explicitly, when $\bar{r}=1$,
$$M_m\leq 2^{4664}c_1^{16}c_2(k)^{256}\left(2h(A)+\tfrac{8}{17}\log[k:\Q]+8\log c_1+128 \log c_2(k) +1503\right)^{512},$$
and when $\bar{r}=2$ or $4$,
\begin{align*}
M_m\leq & (r/4)^{r/2} 2^{48}\cdot c_1^{16}c_2(k)^{256} c_8(A,k)^{17r}\\
 & \cdot \left(16\log c_1+\frac{256}{\bar{r}}\log c_2(k)+16r\log c_8(A,k)+4h(A)+\tfrac{16}{17}\log [k:\Q]+1400\right)^{512/\ol{r}}.
\end{align*}

Where $r$ (resp. $\bar{r}$) is the $\Z$-rank of $\End_k(A)$ (resp. $\End_{\bar{k}}(A)$). 

The constants $c_1$ and $c_2$ are $c_1=4^{11}\cdot 9^{12}$ and
$c_2(k)=7.5\cdot 10^{47}[k:\Q]$, and $c_8(A,k)$ is
$$4^5\cdot 9^8\left(5.04\cdot10^{24}[k:\Q]m_A\left(\tfrac{5}{4}m_A+\log [k:\Q]+\log m_A+60\right)\right)^{8/\bar{r}},$$
where $m_ A$ is $\max(1,h(A))$.
\end{theorem}

\begin{remark}\label{remark:2.1}{\color{white}I need the bullets to be aligned. This is a cheat.}
\begin{itemize}
\item The ranks $r$ and $\bar{r}$ take values in $\{1,2,4\}$ and the inequality $r\leq \bar{r}$ holds.
\item The given explicit bounds in the theorem do indeed not depend on $m$. For ease of notation we will write $M_m=M$.
\end{itemize}
\end{remark}

We sketch a proof of this theorem following the relevant parts in those two papers. As we only focus on abelian surfaces, the bound in the theorem here is slightly sharper and we will emphasize the modifications. We will however need the result of the first lemma also in the case of elliptic curves, so we give the setup for abelian varieties in any dimension.

Let $A$ be a principally polarized abelian variety with polarization $L$ and let $B$ be the abelian variety $A\times A$ principally polarized by $pr_1^*L\otimes pr_2^*L$. Following \cite{MW95}, we denote by $b(B)$ the smallest integer such that for any abelian variety $B'$ defined over $k$, if $B'$ is isogenous to $B$ over $k$, then there exists an isogeny $\phi:B'\rightarrow B$ over $k$ of degree at most $b(B)$. Let $i(A)$ be the class index of the order $\End_k(A)$ defined in \cite{MW95}*{Sec.~2} and let $d(A)$ be the discriminant of $\End_k(A)$ as a $\Z$-module defined in \cite{MW95b}*{Sec.~2}. 

Still letting $B'$ vary over the abelian varieties over $k$ that are isogenous to $B$, let $\dB'$ be the dual abelian variety of $B'$ and let $Z(B')$ be the principally polarizable abelian variety $(B')^4\times (\dB')^4$. We fix a principal polarization on $Z(B')$. In \cite{GR14}*{Sec.~2}, the notion of Rosati involution is generalized to the ring of homomorphisms of abelian varieties and the Rosati involution is used to define a norm on $\End_k(A)$ (resp. $\Hom_k(B,Z(B'))$). Referring to the notation of \cite{GR14}*{Sec.~3}, use $\Lambda$ (resp. $\LamBBp$) to denote $\Lambda(\End_k(A))$ (resp. $\Lambda(\Hom_k(B,Z(B')))$), which is the minimal real number which bounds from above the norms of all elements in some $\Z$-basis of $\End_k(A)$ (resp. $\Hom_k(B,Z(B'))$).  We use $v(A)$ to denote $\vol(\End_k(A))$ with respect to the given norm. 

\begin{lemma}[\cite{MW95}*{Lem.~3.2}]\label{22}
With notation as above, such integers $M_m$ exist satisfying $M_m\leq i(A)b(B)$.
\end{lemma}

From now on, we revert back to the case where $A$ is an abelian surface and the fixed polarization comes from the theta divisor. 

\begin{lemma}\label{23} \label{25} We have $i(A)\leq d(A)^{1/2}=(r/4)^{r/2}v(A).$ and $v(A)\leq \Lambda^r$
\end{lemma}
\begin{proof} The first inequality is \cite{MW95}*{eqn.~2.2} since $A$ is $k$-simple. The second one is by definition (see also the proof of \cite{GR14}*{Lem.~5.3}). The last one is by definition.
\end{proof}

\begin{prop}\label{24}There exists an isogeny $B'\rightarrow B$ over $k$ of degree at most $2^{48}\Lambda_{B, B'}^{16}v(A)^{16}$. 
\end{prop}
\begin{proof}
This is essentially a special case of \cite{GR14}*{Prop.~6.2}. Here we do not need their $\widehat{W}_i$ term since $A$ is principally polarized.
\end{proof}

\begin{lemma}[\cite{Sil}*{Thm.~4.1,~4.2}]\label{FofEnd}
 Given abelian varieties $C, C'$ of dimension $g,g'$ defined over $k$, let $K$ be the smallest field where all the $\ol k$-endomorphisms of $C\times C'$ are defined. Then $[K:k]\leq 4(9g)^{2g}(9g')^{2g'}.$
\end{lemma}
\begin{proof}
This inequality is given by \cite{Sil}*{Thm.~4.2} and \cite{Sil}*{Cor.~3.3}. 
\end{proof}

\begin{lemma}\label{27}
Let $m_A$ and $m_{A,B'}$ denote $\max(1,h(A))$ and $\max(1,h(A), h(B'))$ respectively.
We have
\begin{displaymath}
\Lambda\leq 
\begin{cases}
2 & \textrm{ if } \bar{r}=1, \\
4^5\cdot 9^8\left(5.04\cdot10^{24}[k:\Q]m_A\left(\tfrac{5}{4}m_A+\log [k:\Q]+\log m_A+60\right)\right)^{8/\bar{r}}& \textrm{ if } \bar{r}=2\textrm{ or } 4.
\end{cases}
\end{displaymath}
and
\begin{displaymath}
\LamBBp\leq 4^{11}\cdot 9^{12}\left(4.4\cdot 10^{46}[k:\Q]m_{A,B'}\left(9m_{A,B'}+8\log m_{A,B'}+8 \log[k:\Q]+920\right)\right)^{16/\bar{r}}.
\end{displaymath}
\end{lemma}
\begin{proof}
Recall that $\bar{r}$ denotes the $\Z$-rank of $\End_{\bar{k}}(A)$. To deduce the bound of $\Lambda$, we first study the case $\bar{r}=1$. 
In this case, $\End_{\bar{k}}(A)=\Z$ and by definition the norm of the identity map is $\sqrt{\Tr (\id)}=\sqrt{4}=2$. In other words, $\Lambda=2$.

We postpone the discussion of $\Lambda$ for $\bar{r}=2,4$, since it is a simplified version of the following discussion on the bound of $\LamBBp$.

Let $k_1$ be the field where all the $\ol k$-endomorphisms of $A\times B^{'}$ are defined. Then by Lemma \ref{FofEnd}, we have $[k_1:k]\leq 4\cdot 18^4\cdot 36^8=4^{11}\cdot 9^{12}$.

The estimate of $\LamBBp$ is essentially \cite{GR14}*{Lem.~9.1}. We modify its proof to obtain a sharper bound for this special case. For any complex embedding $\sigma:k_1\rightarrow \C$, we may view $A$ and $Z(B')$ as abelian varieties over $\C$ and let $\Omega_A$ and $\Omega_{Z(B')}$ be the period lattices. As in \cite{GR14}*{Sec.~3}, the principal polarization induces a metric on $\Omega_A$ (resp. $\Omega_{Z(B')}$). Let $\omega_1,\dots, \omega_4$ (resp. $\chi_1,\dots, \chi_{64}$) be a basis of $\Omega_A$ (resp. $\Omega_{Z(B')}$) such that $||\omega_i||\leq \Lambda(\Omega_A)$ (resp.$||\chi_i||\leq \Lambda(\Omega_{Z(B')})$). Let $\omega$ be $(\omega_1,\chi_1,\dots, \chi_{64})\in \Omega_A\oplus (\Omega_{Z(B')})^{64}$ and let $H$ be the smallest abelian subvariety of $A\times (Z(B'))^{64}$ whose Lie algebra (over $\C$) contains $\omega$. Then by \cite{GR14}*{Prop.~7.1, the proof of Prop.~8.2, and the theorem of periods on p. 2095} the bounds
\begin{displaymath}
\Lambda(\Hom_{k_1}(A,Z(B'))\leq (\deg H)^2,
\end{displaymath}
and
\begin{displaymath}
(\deg H)^{1/h}\leq 50[k_1:\Q]h^{2h+6}\max(1,h(H),\log \deg H)||\omega||^2
\end{displaymath}
are satisfied, where $h=\dim H$.
By the proof of \cite{GR14}*{Lem.~8.4}\footnote{where a result of Autissier \cite{autissier}*{Cor.~1.4} is used}, there exists a choice of embedding $\sigma$ such that for any $\epsilon\in (0,1)$,
\begin{displaymath}
||\omega||^2\leq \frac{6}{(1-\epsilon)\pi}\left(16h(A)+8^7h(B')+(16+16^4)\log\left(\frac{2\pi^2}{\epsilon}\right)\right).
\end{displaymath}
By taking $\epsilon=\frac 1{40}$, we have 
$||\omega||^2\leq 5\times 10^6\max(1,h(A),h(B')).$ By \cite{GR14}*{Lem.~8.1}, we have $2\leq h\leq 8/\bar{r}\leq 8$.

Combining the above inequalities, we have the bound
\begin{displaymath}
(\deg H)^{1/h}\leq 1.85\times 10^{28}[k_1:\Q]\max(1,h(A),h(B'))\left(9\max(1,h(A),h(B'))+ \log\deg H +48\right),
\end{displaymath}
where we use the fact (see the discussion in \cite{GR14}*{p.~2096})
\begin{displaymath}
h_F(H)\leq 9\max(1,h(A),h(B'))+ \log\deg H +48.
\end{displaymath}
Then by \cite{GR14}*{Lem.~8.5}\footnote{which is a basic calculating trick and is not specific to bounding the Faltings height}, we have
\begin{displaymath}
\deg H\leq \left(3.7\cdot 10^{28}[k_1:\Q]m_{A,B'}\left(9m_{A,B'}+48+\frac{8}{\bar{r}}\log \left(1.85\cdot 10^{28} [k_1:\Q]\frac {8m_{A,B'}} {\bar{r}}\right)\right)\right)^{8/\bar{r}}.
\end{displaymath}
Then we have (by \cite{GR14}*{Lem.~3.3})
\begin{eqnarray*}
\LamBBp &=& \Lambda(\Hom_k(A,Z(B'))\leq [k_1:k]\Lambda(\Hom_{k_1}(A,Z(B'))\leq [k_1:k](\deg H)^2\\
&\leq &[k_1:k]\left(3.7\cdot 10^{28}[k_1:\Q]m_{A,B'}\left(9m_{A,B'}+48+\frac 8 {\bar{r}}\log \left( 1.85\cdot 10^{28}[k_1:\Q] \frac {8m_{A,B'}} {\bar{r}}\right)\right)\right)^{16/\bar{r}}\\
&\leq & 4^{11}\cdot 9^{12}\left(4.4\cdot 10^{46}[k:\Q]m_{A,B'}\left(9m_{A,B'}+8\log m_{A,B'}+8 \log[k:\Q]+920\right)\right)^{16/\bar{r}}.
\end{eqnarray*}
Now we assume that $\bar{r}=2$ or $4$. In this case we cannot compute $\Lambda$ so we apply the same strategy as for the bound on $\Lambda_{B,B'}$. The proof is practically identical, but the bounds are different. In this case we bound the degree $[k_1:k]\leq 4\cdot 18^8$ and there exists an abelian subvariety $H$ of $A\times A^4$ over $k_1$ such that the bounds
\begin{displaymath}
\Lambda \leq [k_1:k](\deg H)^2
\end{displaymath}
and
\begin{displaymath}
\deg H\leq \left(100\cdot 4^{19}\cdot 9^8\cdot 1063[k:\Q]m_A\left(5m_A+4\log [k:\Q]+4\log m_A+240\right)\right)^{8/\bar{r}}
\end{displaymath}
are satisfied.
Combining these two inequalities together, we obtain the bound for $\Lambda$.
\end{proof}

\begin{proof}[Proof of Theorem \ref{effFal}] 
The proof is a combination of applying the lemmas above. We start by bounding the smallest degree of isogenies from $B'$ to $B$ (for which we use the notation $b(B)$). Let $\phi:B'\rightarrow B$ be an isogeny of the smallest degree $d$. We want to bound $d$ in terms of $h(A)$ and $[k:\Q]$. First, we notice that 
\begin{displaymath}
h(B')\leq h(B)+\tfrac{1}{2}\log \deg(\phi)=2h(A)+\tfrac{1}{2}\log \deg(\phi)=2h(A)+\tfrac{1}{2}\log d.
\end{displaymath}
Then $m_{A,B'}=\max(1,h(A),h(B'))\leq 2h(A)+\tfrac{1}{2}\log d +7,$ since $h(A)\geq -3$ holds. Then by Lemma \ref{27} and the fact $m_{A,B'}\geq \log m_{A,B'}$, we have 
\begin{equation}\label{equation:bound_LambdaBBprime}
\Lambda_{B, B'} \leq c_1 \left( c_2(k)\left(c_3(A,k)+ \tfrac{1}{2} \log d\right)^2\right)^{\frac{16}{\bar{r}}},
\end{equation}
where $\bar{r}=1,2$ or $4$ and the constants are defined as

\begin{displaymath}
\left\{
\begin{aligned}
& c_1 = 4^{11}\cdot 9^{12},\\
& c_2(k)=7.5\cdot 10^{47}[k:\Q],\\
& c_3(A,k)=2h(A)+\tfrac{8}{17}\log [k:\Q]+\tfrac{1039}{17}.
\end{aligned}
\right.
\end{displaymath}

We furthermore introduce the constants
\begin{displaymath}
\left\{
\begin{aligned}
& c_4(A,k)=\sqrt{c_2(k)}c_3(A,k),\\
& c_5(k)= \frac{\sqrt{c_2(k)}}{2},\\
& c_6(A,k)=2^{48}\cdot c_1^{16}\cdot \Lambda^{16r},\\
\end{aligned}
\right.
\end{displaymath}
and we rewrite inequality \eqref{equation:bound_LambdaBBprime} as:
\begin{displaymath}
\Lambda_{B, B'} \leq c_1 [c_4(A,k) + c_5(k) \log d]^{\frac{32}{\bar{r}}}.
\end{displaymath}
Then by Lemmas \ref{25} and \ref{24}, we have
\begin{equation}\label{equation: d_leq}
d=\deg \phi \leq 2^{48}\LamBBp^{16}v(A)^{16}\leq 2^{48}\LamBBp^{16}\Lambda^{16r}\leq c_6(A,k)\left[c_4(A,k) + c_5(k) \log d\right]^{\frac{32\cdot 16}{\bar{r}}}.
\end{equation}
We define $c_7(A,k)=2^{48}\cdot c_1^{16}\cdot c_8(A,k)^{16r}$
with $c_8(A,k)$ defined as
\begin{displaymath}
c_8(A,k)=\begin{cases}
2 & \mathrm{if}\ \bar{r}=1, \\
4^5\cdot 9^8\left(5.04\cdot10^{24}[k:\Q]m_A\left(\tfrac{5}{4}m_A+\log [k:\Q]+\log m_A+60\right)\right)^{8/\bar{r}}& \mathrm{if}\ \bar{r}=2,4.
\end{cases}
\end{displaymath}
Then by Lemma \ref{27}, $c_6(A,k)\leq c_7(A,k)$.
We rewrite inequality \eqref{equation: d_leq} as 
\begin{displaymath}
d^{\frac{\bar{r}}{32\cdot 16}}\leq u(A,k)\left(\tfrac{\bar{r}}{32\cdot 16}\log d +v(A,k)\right),
\end{displaymath}
where 
\begin{displaymath}
\left\{
\begin{aligned}
& u(A,k)=c_7(A,k)^{\frac{\bar{r}}{32\cdot 16}}c_5(A,k)\cdot\frac{32\cdot 16}{\bar{r}},\\
& v(A,k)=\frac{c_4(A,k)\bar{r}}{32\cdot 16c_5(A,k)}.
\end{aligned}
\right.
\end{displaymath}

Then by \cite{GR14}*{Lem.~8.5}, we have
\begin{displaymath}
d^{\frac{\bar{r}}{32\cdot 16}}\leq 2u(A,k)[\log u(A,k)+v(A,k)].
\end{displaymath}
Define 
\begin{displaymath}
C(A,k) = 2 u(A,k)[\log u(A,k) + v(A,k)],
\end{displaymath}
which only depends on $h(A)$ and $[k:\Q]$. Then we find
\begin{displaymath}
b(B)\leq C(A,k)^{\frac{32\cdot 16}{\bar{r}}}.
\end{displaymath}
By Lemma \ref{22} and \ref{23}, we obtain:
\begin{displaymath}
M\leq i(A)b(B)\leq (r/4)^{r/2}c_8(A,k)^{r}C(A,k)^{\frac{32\cdot 16}{\bar{r}}}.
\end{displaymath}

Using $r\leq \bar{r}$, in the case $\bar{r}=1$ we find
$$M\leq 2^{4664}c_1^{16}c_2(k)^{256}\left(2h(A)+\tfrac{8}{17}\log[k:\Q]+8\log c_1+128 \log c_2(k) +1503\right)^{512},$$

and in the case $\bar{r}=2$ or $4$ we find
\begin{displaymath}
\begin{array}{ll}
M\leq & (r/4)^{r/2} 2^{48}\cdot c_1^{16}c_2(k)^{256}\\
& \cdot \left(4^5\cdot 9^8\left(5.04\cdot10^{24}[k:\Q]m_A\left(\tfrac{5}{4}m_A+\log [k:\Q]+\log m_A+60\right)\right)^{8/\bar{r}}\right)^{17r}\\
 & \cdot \left(16\log c_1+\frac{256}{\bar{r}}\log c_2(k)+16r\log c_8(A,k)+4h(A)+\tfrac{16}{17}\log [k:\Q]+1400\right)^{512/\bar{r}}.
\end{array}
\end{displaymath}

The constants $c_1$, $c_2(k)$ and $c_8(A,k)$ only depend on the Faltings height and the degree of the field extension $[k:\Q]$, justifying Remark \ref{remark:2.1}.
\end{proof}

\subsection{$k$-isogenous to product of elliptic curves}\label{prodEC}

Let $E$ and $E'$ be elliptic curves over $k$ such that $A$ is isogenous to $E\times E'$ over $k$. 
We also assume in this subsection that if $E$ is isogenous to $E'$ over $\ol{k}$, then they are isogenous over $k$. Hence we may choose $E'$ such that if $E'\neq E$, then $E'$ is not isogenous to $E$ over $k$. Notice that $E\times E'$ can be endowed with a principal polarization and we will fix the polarization to be the one induced by the line bundle $pr_1^*L_1\otimes pr_2^*L_2$, where $pr_1:E\times E'\rightarrow E$ and $pr_2:E\times E'\rightarrow E'$ are the projections and where $L_1$ and $L_2$ are the line bundles inducing the natural polarizations on $E$ and $E'$. Together with our assumption that $A$ is principally polarized, this improves the bounds in \cite{GR14}. In this subsection, the polarization on the product of polarized abelian varieties is always taken to be the natural product of polarizations.

\begin{theorem}[special case of \cite{GR14}*{Thm.~1.4} with slightly better bound] \label{FalEC}
The minimal degree of the isogeny between $A$ and $E\times E'$ is at most
\begin{displaymath}
C_1(h(A), [k:\Q])=230^8\cdot 3^8\cdot 4^{159}\cdot [k:\Q]^8\left(\tfrac{17}{8} h(A)+\tfrac{5}{6}\log[k:\Q]+30.8\right)^{16}.
\end{displaymath}
\end{theorem}

Although we focus on the case when $A$ is principally polarized, the following theorem for non-principally polarized situation will also be used later.
\begin{theorem}[special case of \cite{GR14}*{Thm.~1.4} with slightly better bound] \label{FalECnp}
We assume that if $E=E'$, then $E$ is without complex multiplication.
Let $B$ be a polarized abelian surface isogenous to $E\times E'$. Then the minimal degree of the isogeny from $B$ to $E\times E'$ is bounded from above by a constant $C_2(h(E),h(E'), [k:\Q])$. More explicitly, let $m=\max(1,h(E),h(E'))$ and write $d=[k:\Q]$, then
\begin{itemize}
\item if $E\neq E'$ and at least one of them is without complex multiplication, we have
\begin{align*}
C_2=& 1.74\cdot 10^{571}d^{66}\left(m+\tfrac{1}{2}\log d\right)^2\\
&\cdot\left(525100m+4.42\cdot 10^8+8.67\cdot 10^6\log d+2^{18}\log\left(m+\tfrac{1}{2}\log d\right)\right)^{128},
\end{align*}
\item if $E\neq E'$ and both of them have complex multiplication, we get
\begin{align*}
C_2=& 8.78\cdot 10^{342}d^{36}\left(m+\tfrac{1}{2}\log d\right)^4\\
&\cdot \left(273m+2449\log d+272\log\left(m+\tfrac{1}{2}\log d\right) + 8.79\cdot 10^4\right)^{64},
\end{align*}

\item if $E=E'$ without complex multiplication, we have
\begin{displaymath}
C_2=3.61\cdot 10^{309}d^{32}\left(273m+2177\log d+8.26\cdot 10^4\right)^{64}.
\end{displaymath}
\end{itemize}
\end{theorem}

We use the same notation as before and need the following lemmas to prove the theorems.
\begin{lemma}[see also \cite{GR14}*{Lem.~3.2, Prop.~6.2}]\label{lemma: isogenyI}
There exist isogenies $A\rightarrow E\times E'$ and $E\times E'\rightarrow A$ over $k$ of degree at most
\begin{displaymath}
4^2\Lambda(\Hom_k(A,E\times E')).
\end{displaymath}
There exists an isogeny $B\rightarrow E\times E'$ over $k$ of degree at most
\begin{displaymath}
\begin{cases}
32^4 \Lambda\left(\Hom_k\left(E\times E', Z(B)\right)\right)^8v(E)^2v(E')^2,& \mathrm{when}\ E\neq E';\\
32^4 \Lambda\left(\Hom_k\left(E, Z(B)\right)\right)^8v(E)^8,& \mathrm{when}\ E=E'.
\end{cases}
\end{displaymath}
\end{lemma}
\begin{proof}
Since both $A$ and $E\times E'$ are principally polarized, the Rosati involution induces the isometry $\Hom_k(A,E\times E')\cong \Hom_k(E\times E',A)$ and hence there is the equality of suprema $\Lambda(\Hom_k(A,E\times E'))=\Lambda( \Hom_k(E\times E',A))$. Then the first assertion follows directly from \cite{GR14}*{Lem.~3.2} by noticing that $h^0=1$.

The second assertion follows from the proof of \cite{GR14}*{Prop.~6.2} by noticing that $E$ and $E'$ are naturally isomorphic to their duals. 
\end{proof}

\begin{lemma}\label{lemma: 2.11}
Assume that all the endomorphisms of $E\times E'\times A$ are defined over $k_1$. Then there exists an abelian subvariety $H\subset E\times A^4$ satisfying
\begin{enumerate}
\item $\Lambda(\Hom_{k_1}(E,A))\leq (\deg H)^2,$ where the degree is with respect to the natural polarization on $E\times A^4$; and\label{2.11.1}
\item the degree of $H$ is at most
\begin{displaymath}
\left(230[k_1:\Q]\cdot 4^{16}\left(h(E)+16h(A)+79\right)\left(\log[k_1:\Q]+3h(A)+\tfrac{3}{8}h(E)+41\right)\right)^4.
\end{displaymath}
\end{enumerate}
The same result holds for $E'$.
\end{lemma}
\begin{proof}
We follow the proof of \cite{GR14}*{Prop.~7.1, Prop.~8.2}. For any complex embedding $\sigma:k_1\rightarrow \C$, we may view $A$ and $E$ as abelian varieties over $\C$ and let $\Omega_A, \Omega_E$ be the period lattices. The polarizations induce metrics on the lattices (see \cite{GR14}*{Sec.~3}).
Let $\omega_1,\cdots, \omega_4$ be a basis of $\Omega_A$ such that for each of them $||\omega_j||\leq \Lambda(\Omega_A)$ holds and let $\chi_1,\chi_2$ be a basis of $\Omega_E$ such that $||\chi_i||\leq \Lambda(\Omega_E)$ holds for $i=1,2$.
Let $\omega$ be $(\chi_1,\omega_1,\cdots, \omega_4)$ and $H$ be the smallest abelian subvariety with tangent space containing $\omega$. From now on, we fix the complex embedding to be one for which $\frac{4}{\lambda(\Omega_E)^2}+\frac{64}{\lambda(\Omega_A)^2}$ is smallest, where $\lambda$ is the minimal length of non-zero elements in the lattice.
Then the same argument as in the proof of \cite{GR14}*{Prop.~8.2} shows that the conditions in \cite{GR14}*{Prop.~7.1} hold and hence we obtain \eqref{2.11.1}.

To bound $\deg H$, we apply the theorem of periods (see \cite{GR14}*{pp.~2095}) (and ignore the contributions of other embeddings):
\begin{displaymath}
\deg H^{1/h}\leq 50[k_1:\Q]h^{2h+6}\max(1,h(H),\log \deg H)||\omega||^2,
\end{displaymath}
where $h=\dim H$. From \cite{GR14}*{Lem.~8.1}, we have $h\leq \nu(A)\leq 4$.

Moreover, by a result of Autissier
\cite{autissier}*{Cor.~1.4}, (taking $\epsilon$ to be $\frac 16$,) we have
\begin{displaymath}
||\omega||^2 \leq \frac 4 {\lambda(\Omega_E)^2}+\frac {64}{\lambda(\Omega_A)^2} \leq \tfrac{36}{5\pi}(4h(E)+64h(A)+316).
\end{displaymath}

On the other hand, using the lower bound on Faltings' height, we have (see \cite{GR14a}*{pp.~352})
\begin{align*}
h(H) & \leq h(E\times A^4)+\log \deg H+\tfrac{3}{2}(\dim(E\times A^4)-\dim H)\\
& \leq h(E)+4h(A)+\log \deg H +12.
\end{align*}
Combining the above inequalities, we arrive at
\begin{align*}
(\deg H)^{1/4} & \leq 115[k_1:\Q] h^{2h+6}\left(4h(E)+64h(A)+316\right)\max(1,h(H), \log \deg H)\\
& \leq 115[k_1:\Q] 4^{14}\left(4h(E)+64h(A)+316\right)\left(h(E)+4h(A)+15+\log \deg H\right).
\end{align*}
By \cite{GR14}*{Lem.~8.5}, we have
\begin{align*}
\deg H \leq &\left(230[k_1:\Q]\cdot 4^{16}\left(h(E)+16h(A)+79\right)\right)^4\\
& \cdot \left(\log([k_1:\Q]\left(h(E)+16h(A)+79)\right)+h(A)+h(E)/4+31\right)^4\\
\leq &\left(230[k_1:\Q]\cdot 4^{16}\left(h(E)+16h(A)+79\right)\left(\log[k_1:\Q]+3h(A)+\tfrac{3}{8}h(E)+41\right)\right)^4.
\end{align*}
\end{proof}

\begin{lemma}\label{lemma: 2.12}
Assume that all the endomorphisms of $E\times E'\times B$ are defined over $k_1$. Then there exists an abelian subvariety $H\subset E\times Z(B)^{32}$ satisfying
\begin{enumerate}
\item $\Lambda(\Hom_{k_1}(E,Z(B)))\leq (\deg H)^2;\label{2.12.1} and$
\item the degree of $H$ with respect to the polarization on $E\times Z(B)^{32}$ is at most
\begin{displaymath}
\left(230[k_1:\Q]4^{15}\left(4h(E)+2^{18}h(B)+1.26\times10^6\right)\left(0.251h(E)+72h(B)+254.5+\log [k_1:\Q]\right)\right)^4.
\end{displaymath}
\end{enumerate}
The same result holds for $E'$.
Moreover, the degree of $H$ is bounded from above by 
\begin{displaymath}
\left(230[k_1:\Q]2^{11}\left(4h(E)+2^{18}h(B)+1.26\times 10^6\right)\left(0.51h(E)+136h(B)+\log [k_1:\Q]+434\right)\right)^2
\end{displaymath}
when either $E=E'$ holds, or when both $E$ and $E'$ have complex multiplication.
\end{lemma}
\begin{proof}
The idea is the same as in the proof of Lemma \ref{lemma: 2.11}. Let $\chi_1,\chi_2$ be a basis of $\Omega_E$ such that $||\chi_i||\leq \Lambda(\Omega_E)$ for $i=1,2$ and $\omega_1,\cdots, \omega_{32}$ a basis of $\Omega_{Z(B)}$ such that $||\omega_j||\leq \Lambda(\Omega_{Z(B)})$ for $j=1,\ldots,32$.
Let $\omega$ be $(\chi_1,\omega_1,\cdots,\omega_{32})$ in $\Omega_{E\times Z(B)^{32}}$ and $H$ be the smallest abelian subvariety with tangent space containing $\omega$. Then we obtain \eqref{2.12.1} in a way similar to the proof of the previous lemma.

We have
\begin{displaymath}
(\deg H)^{1/h}\leq 50[k_1:\Q]h^{2h+6}\max(1,h(H),\log \deg H)||\omega||^2_\sigma,
\end{displaymath}
where $h=\dim H\leq \nu(B)\leq 4.$ 
Moreover, we have (by a result of \cite{autissier}*{Cor.~1.4})
\begin{align*}
||\omega||^2 &=||\chi_1||^2+\sum ||\omega_j||^2\leq \Lambda(\Omega_E)^2+32\Lambda(\Omega_{Z(B)})^2\leq \frac{4}{\lambda(\Omega_E)^2}+\frac {32^3}{\lambda(\Omega_{Z(B)})^2}\\
&\leq \frac{6}{(1-\epsilon)\pi}\left(4h(E)+2\log\left(\frac{2\pi^2}{\epsilon}\right)+32^3\cdot 8 h(B)+32^3\cdot 8\log\left(\frac{2\pi^2}{\epsilon}\right)\right),
\end{align*}
for any $\epsilon\in (0,1)$. 
On the other hand there is the bound
\begin{align*}
h(H) & \leq h(E\times Z(B)^{32})+\log \deg H +\tfrac{3}{2} (\dim(E\times Z(B)^{32})-\dim H)\\
&\leq h(E)+2^8h(B)+\log \deg H+3\cdot 2^8.
\end{align*}
Then we have (by taking $\epsilon=\frac{1}{6}$)
\begin{align*}
(\deg H)^{1/h}\leq &115[k_1:\Q]\cdot 4^{14}\left(h(E)+2^8h(B)+\log \deg H+3(2^8+1)\right)\\
&\cdot(4h(E)+2^{18}h(B)+1.26\times 10^6).
\end{align*}

We conclude 
\begin{align*}
\deg H\leq \Big(230&[k_1:\Q]\cdot 4^{15}\left(4h(E)+2^{18}h(B)+1.26\times10^6\right)\\
& \cdot \left(0.251h(E)+72h(B)+\log [k_1:\Q]+254.5\right)\Big)^4.
\end{align*}
If either $E$ and $E'$ are equal, or if both of them have complex multiplication, then we notice $h\leq \nu(B)\leq 2$ and hence the same argument as above provides the better bound
\begin{align*}
\deg H\leq \Big(230&[k_1:\Q]\cdot 2^{11}\left(4h(E)+2^{18}h(B)+1.26\times 10^6\right)\\
&\cdot \left(0.51h(E)+136h(B)+\log [k_1:\Q]+434\right)\Big)^2.
\end{align*}
\end{proof}

\begin{proof}[Proof of Theorem \ref{FalEC}]
Our situation satisfies $[k_1:k]\leq 4$ since all endomorphism of an elliptic curve are defined over some quadratic extension. Then by the above lemmas \ref{lemma: isogenyI} and \ref{lemma: 2.11}, we bound the minimal degree $D$ of the isogeny between $A$ and $E\times E'$ by
\begin{align*}
D & \leq 4^2\Lambda(\Hom_k(A,E\times E'))\\
& \leq 4^2 [k_1:k]\Lambda(\Hom_{k_1}(A,E\times E'))\\
& \leq 4^3 \left(230[k:\Q]4^{17}\left(m+16h(A)+79\right)\left(\log[k:\Q]+3h(A)+\tfrac{3}{8}m+42.5\right)\right)^8\\
& \leq 230^8\cdot 3^8\cdot4^{127} \cdot [k:\Q]^8\cdot\left(17h(A)+115+\tfrac{8}{3} \log [k:\Q] +\tfrac{1}{2}\log D\right)^{16},
\end{align*}
where $m=\max(h(E),h(E'))\leq h(A)+\frac{1}{2} \log D +\frac{3}{2}$. In order to arrive at the third line, we have used the bound on $[k_1:k]$ stated in the first line of the proof. To arrive at the fourth line, we bound both factors within the parentheses by the same factor that appears in the fourth line, taking $\left(\frac{3}{8}\right)^8$ out of the parentheses and into the leading factor.

By \cite{GR14}*{Lem.~8.5}, we have
\begin{displaymath}
D\leq 230^8\cdot 3^8\cdot 4^{159}\cdot [k:\Q]^8\left(\tfrac{17}{8} h(A)+\tfrac{5}{6} \log[k:\Q]+30.8\right)^{16}.
\end{displaymath}
\end{proof}

In the proof of Theorem \ref{FalECnp}, we need the following lemma.
\begin{lemma}\label{lemma: 2.13}
Let $E$ be an elliptic curve over $k$ with complex multiplication. We have 
\begin{displaymath}
v(E)=\vol(\End_k(E))\leq 52[k:\Q]\max(1,h(E)+\tfrac{1}{2}\log[k:\Q]).
\end{displaymath}
\end{lemma}
\begin{proof}
This lemma can be deduced from the proof of \cite{GR14}*{Prop.~10.1}. We may assume that all endomorphisms of $E$ are defined over $k$ since otherwise $v(E)=\sqrt{2}$ holds and the lemma holds trivially. Let $\psi\in \End_k(E)$ as defined in the proof of \cite{GR14}*{Prop.~10.1}. Their proof shows $v(E)\leq 2\sqrt{\deg \psi}$ and 
\begin{displaymath}
\deg \psi \leq 668[k:\Q]^2\max\left(1,h(E)+\tfrac{1}{2}\log [k:\Q]\right)^2.
\end{displaymath}
Combining these two inequalities, we obtain the desired inequality.
\end{proof}

\begin{proof}[Proof of Theorem \ref{FalECnp}]
Let $D$ be the minimal possible degree of an isogeny $B\rightarrow E\times E'$. 

If $E\neq E'$ holds and at least one of them, say $E'$, is without complex multiplication, then $v(E')=\sqrt{2}$ and $[k_1:k]\leq 2$ hold. By the lemmas \ref{lemma: isogenyI} and \ref{lemma: 2.12}, we have 
\begin{align*}
D\leq  &32^4 \Lambda(\Hom_k(E\times E', Z(B)))^8v(E)^2v(E')^2 \\
\leq  & 32^4[k_1:k]^8\Lambda(\Hom_{k_1}(E\times E', Z(B))^8v(E)^2v(E')^2\\
\leq  & \left(230[k_1:\Q]4^{15}\left(4m+2^{18}h(B)+1.26\times10^6\right)\left(0.251m+72h(B)+254.5+\log [k_1:\Q]\right)\right)^{64}\\
& \cdot  2^{28}v(E)^2v(E')^2\\
\leq & \left(230[k_1:\Q]4^{15}\left((2^{19}+4)m+2^{17}\log D+1.26\times10^6\right)\right)^{64}\\
& \cdot\left(\left(144.251m+36\log D+254.5+\log [k_1:\Q]\right)\right)^{64}\cdot  2^{28}v(E)^2v(E')^2,
\end{align*}
where again we write $m=\max(1,h(E),h(E'))$ and the last inequality uses the upper bound $h(B)\leq 2m+\tfrac{1}{2}\log D$. We consider both factors which are raised to the 64th power. By extracting a factor 3640 from the first one, we can bound both by the same expression. Doing so, we arrive at
\begin{displaymath}
D \leq \left(1.36\cdot 10^8[k:\Q]\left(525100m+2^{17}\log D+1.26\times 10^6+7280\log [k:\Q]\right)^2\right)^{64} \cdot  2^{29}v(E)^2
\end{displaymath}
and hence 
\begin{align*}
D \leq & 2^{157}v(E)^2(1.36\cdot 10^8[k:\Q])^{64}\\
& \cdot\left(525100m+4.41\cdot 10^8+7280\log [k:\Q]+2^{24}\left(\tfrac{1}{2}\log[k:\Q]+\tfrac{1}{64}\log(v(E))\right)\right)^{128}.
\end{align*}

If $E\neq E'$ holds and both $E$ and $E'$ have complex multiplication, then $[k_1:k]\leq 4$ and we have
\begin{align*}
D\leq & 32^4 \Lambda(\Hom_k(E\times E', Z(B)))^8v(E)^2v(E')^2 \\
\leq & 32^4[k_1:k]^8\Lambda(\Hom_{k_1}(E\times E', Z(B))^8v(E)^2v(E')^2\\
\leq & 4^{194}v(E)^2v(E')^2\Big(230[k_1:\Q](4m+2^{18}h(B)+1.26\cdot 10^6)\cdot\\
& \cdot(0.51m+136h(B)+\log [k_1:\Q]+434)\Big)^{32}\\
\leq & 4^{194}v(E)^2v(E')^2\left(1.78\cdot 10^6[k:\Q]\left(273m+68\log D+\log [k:\Q]+654\right)^2\right)^{32}.\\
\end{align*}
Hence we have
\begin{displaymath}
D\leq 4^{226}v(E)^2v(E')^2(1.78\cdot 10^6[k:\Q])^{32}(273m+2177\log [k:\Q]+136\log (v(E)v(E'))+8.68\cdot 10^4)^{64}.
\end{displaymath}
For these two cases, we conclude by Lemma \ref{lemma: 2.13} bounding $v(E)$ and $v(E')$ in terms of $m$.

If $E=E'$ holds, then so do $k_1=k$ and $v(E)=\sqrt{2}$. We have
\begin{align*}
D & \leq 32^4 \Lambda(\Hom_k(E, Z(B)))^8v(E)^8\\
& \leq 2^{20} \Lambda(\Hom_{k_1}(E,Z(B))^8v(E)^8\\
& \leq 4^{180}v(E)^8(230[k_1:\Q](4m+2^{18}h(B)+1.26\cdot 10^6)(0.51m+136h(B)+\log [k_1:\Q]+434))^{32}\\
& \leq 4^{182}(4.44\cdot 10^5[k:\Q](273m+68\log D+\log [k:\Q]+654)^2)^{32}.
\end{align*}
Hence we have
\begin{displaymath}
D \leq 4^{214}(4.44\cdot 10^5[k:\Q])^{32}(273m+2177\log [k:\Q]+8.26\cdot 10^4)^{64}.
\end{displaymath}
\end{proof}

\subsection{$k$-simple but not geometrically simple}

\begin{theorem}\label{AtoEC}
Let $A$ be a principally polarized abelian surface that is not geometrically simple. Then there is a field extension $k\subset k_1$ such that there exists a $k_1$-isogeny between $A$ and a product of two elliptic curves over $k$ of degree at most 
\begin{displaymath}
C_3(h(A),[k:\Q])=230^8\cdot 3^8\cdot 4^{167}\cdot 18^{64}\cdot [k:\Q]^8\left(\tfrac{17}{8}h(A)+\tfrac{5}{6} \log [k:\Q]+49.2\right)^{16}.
\end{displaymath}
\end{theorem}
\begin{proof}
Let $k_1$ be the field where all the endomorphisms of $A$ are defined. Then there exist elliptic curves $E$ and $E'$ such that $A$ is isogenous to $E\times E'$ over $k_1$. The assertion follows from Theorem \ref{FalEC} if we have $[k_1:k]\leq 4\cdot 18^8$. This follows from applying Lemma \ref{FofEnd} to $A$.
\end{proof}

\subsection{related results for elliptic curves}
In this subsection, we discuss variants of the effective Faltings' theorem for elliptic curves. All the results are special cases of the main theorems of \cite{GR14} with possibly better bounds. This completes the results from Theorem \ref{FalECnp} by adding the case where $A$ is geometrically isogenous to a product of two equal elliptic curves having complex multiplication.
\begin{theorem}[\cite{GR14}*{Prop.~10.1}]\label{CMECmax}
\label{theorem: isogenyI}
Let $E$ be an elliptic curve over $k_1$, that when base changed to $\ol k$ has complex multiplication by $K$. Then there is a finite field extension $k_1\subset k_2$ and an elliptic curve $E''$ over $k_2$, isogenous to $E_{k_2}$ satisfying:
\begin{enumerate}
\item
$\End_{k_2}(E'')=\OO_K$;
\item there exists an isogeny $\phi$ over $k_2$ between $E$ and $E''$ with
\begin{displaymath}
\deg \phi \leq 30[k_1:\Q] \max\left(1,h(E)+\tfrac{1}{2} \log [k:\Q]\right);
\end{displaymath}
\item $[k_2:k_1]\leq 2 (\deg \phi)^2.$
\end{enumerate}
\end{theorem}

\begin{cor}\label{rk4deg}
Let $A$ be an abelian surface with N\'eron--Severi rank $4$. Then there is a field extension $k\subset k_2$ and an elliptic curve $E''$ over $k_2$ such that there is an isogeny between $A$ and $E''\times E''$ of degree bounded from above by 
\begin{displaymath}
C_4(h(A),[k:\Q])=225\cdot 4^2\cdot 18^{16}C_3[k:\Q]^2\left(h(A)+\frac{\log C_3}2+\log [k:\Q]+25\right)^2,
\end{displaymath}
and the degree of field extension $[k_2:k]$ is bounded by
\begin{displaymath}
C_5(h(A),[k:\Q])=1800\cdot 4^2\cdot 18^{24}[k:\Q]^2\left(h(A)+\frac{\log C_3}2+\log [k:\Q]+25\right)^2,
\end{displaymath}
where $C_3$ is the constant depending on $h(A)$ and $[k:\Q]$ from Theorem \ref{AtoEC}.
\end{cor}
\begin{proof}
Combining Theorems \ref{AtoEC} and \ref{CMECmax} and noticing $2h(E)\leq h(A)+\frac 12 \log C_3$ (and the right hand side is always greater than $2$), we conclude that the degree of the isogeny is bounded by 
\begin{align*}
C_4 &:= C_3\cdot\left(30[k_1:\Q]\max\left(1,h(E)+\tfrac{1}{2}\log [k_1:\Q]\right)\right)^2\\
& \leq 900\cdot 4^2\cdot 18^{16}C_3[k:\Q]^2\left(\frac{h(A)+(\log C_3)/2}2+\tfrac{1}{2} (\log [k:\Q]+25)\right)^2\\
&=225\cdot 4^2\cdot 18^{16}C_3[k:\Q]^2\left(h(A)+\frac{\log C_3}2+\log [k:\Q]+25\right)^2.
\end{align*}
Furthermore, we have
\begin{align*}
[k_2:k] & =[k_2:k_1][k_1:k]\leq 8\cdot 18^8\left(30[k_1:\Q]\max(1,h(E)+\tfrac{1}{2}\log [k_1:\Q])\right)^2\\
& \leq 1800\cdot 4^2\cdot 18^{24}[k:\Q]^2\left(h(A)+\frac{\log C_3}2+\log [k:\Q]+25\right)^2.
\end{align*}
\end{proof}

\begin{theorem}\label{effFalEC}
Let $E$ and $E'$ be elliptic curves over a number field $k$.
For any positive integer $m$, let $M_m$ be the smallest positive integer that kills the cokernel of the map $\Hom_k(E,E')\rightarrow \Hom_\Gamma(E_m,E'_m)$.
Then there exists an explicitly computable upper bound on $M_m$ depending only on $h(E)$, $h(E')$, and $[k:\Q]$. Moreover, $C_2(h(E),h(E'), [k:\Q])$ from Theorem \ref{FalECnp} suffices.
\end{theorem}
\begin{proof}
When $E$ and $E'$ are isogenous and without complex multiplication, we have $i(E)=1$ and we arrive at the desired statement by applying Lemma \ref{22} and Theorem \ref{FalECnp}.

When $E$ is not isogenous to $E'$, we prove a variant of Lemma \ref{22}. Let $f$ be an element of $\Hom_\Gamma(E_m,E'_m)$ and let $G\subset E_m\times E'_m$ be its graph. If we write $B$ for the quotient of $E\times E'$ by $G$, then $B$ is defined over $k$. By Theorem \ref{FalECnp}, there exists an isogeny $B\rightarrow E\times E'$ of some degree $b\leq C_2(h(E),h(E'), [k:\Q])$.

Consider the composite map $\chi:E\times E'\rightarrow (E\times E')/G=B\rightarrow E\times E'$. By the proof of \cite{MW95}*{Lem.~3.1} we have $b\ker \chi\subset G\subset \ker \chi$. We conclude the proof by proving that the map $f$ is killed by $b$. Since $E$ and $E'$ are not isogenous, we may write $\chi$ as $(\alpha, \beta)$ where $\alpha\in \End(E)$ and $\beta\in \End(E')$. Given any $y\in E_m$, we only need to prove $bf(y)=0\in E'_m$. By definition, $(y,f(y))\in G\subset \ker \chi$ and hence $\beta f(y)=0$. On the other hand, $(0,bf(y))\in b(\ker \alpha \times \ker \beta)= b \ker \chi \subset G$ and since $G$ is the graph of $f$, this implies $bf(y)=f(0)=0$. Since $f$ is arbitrary, we conclude that the cokernel of the map $\Hom_k(E,E')\rightarrow \Hom_\Gamma(E_m,E'_m)$ is killed by some $b\leq C_2(h(E),h(E'), [k:\Q])$.
\end{proof}

\section{Effective computations of the N\'eron--Severi lattice as a Galois module}\label{Alg_part1}

Our goal of this section is to prove the following theorem:
\begin{theorem}
There is an explicit algorithm that takes input a smooth projective curve $C$of genus $2$ defined over a number field $k$, and outputs a bound of the algebraic Brauer group $\Br_1(X)/\Br_0(X)$ where $X$ is the Kummer surface associated to the Jacobian $\Jac(C)$.
\end{theorem}

A general algorithm to compute N\'eron--Severi groups for arbitrary projective varieties is developed in \cite{PTV}, so here we consider algorithms specialized to the Kummer surface $X$ associated to a principally polarized abelian surface $A$.

\subsection{The determination of the N\'eron--Severi rank of $A$}
\begin{theorem}\label{32}
The following is a complete list of possibilities for the rank $r$ of $\NS(\overline{A})$. For any prime $\frakp$ we denote by $r_\frakp$ the reduction of $r$ modulo $\frakp$.
\begin{enumerate}
\item When $A$ is geometrically simple, we consider $D=\End_{\bar{k}}(A)\otimes \Q$, which has the following possibilities:
\begin{enumerate}
\item $D=\Q$ and $r=1$. There exists a density one set of primes $\frakp$ with $r_{\frakp}=2$.
\item $D$ is a totally real quadratic field. Then $r=2$ and there exists a density one set of primes $\frakp$ with $r_{\frakp}=2$.
\item $D$ is a indefinite quaternion algebra over $\Q$. Then $r=3$ and there exists a density one set of primes $\frakp$ with $r_{\frakp}=4$. 
\item $D$ is a degree $4$ CM field. Then $r=2$ and there exists a density one set of primes $\frakp$ with $r_{\frakp}=2$. In fact this holds for the set of $\frakp$'s such that $A$ has ordinary reduction at $\frakp$.
\end{enumerate}

\item When $A$ is isogenous over $\bar{k}$ to $E_1\times E_2$ for two elliptic curves. Then

\begin{enumerate}
\item if $E_1$ is isogenous to $E_2$ and CM, then $r=4$ and $r_{\frakp}=4$ for all ordinary reduction places.
\item if $E_1$ is isogenous to $E_2$ but not CM, then $r=3$ and $r_{\frakp}=4$ for all ordinary reduction places.
\item if $E_1$ is not isogenous to $E_2$, then $r=2$ and there exists a density one set of primes $\frakp$ such that $r_{\frakp}=2$.
\end{enumerate}

\end{enumerate}
Notice that for all the above statements, by an abuse of language, being density one means there exists a finite extension of $k$ such that the primes are of density one with respect to this finite extension.
\end{theorem}

 \begin{proof}
 We apply \cite{MumAV}*{p.~201 Thm. 2 and p.208} (and the remark on p. 203 referring to the work of Shimura) to obtain the list of the rank $r$. When $A$ is geometrically simple, we can only have $A$ of type I, II, and IV (in the sense of the Albert's classification). In the case of Type I, the totally real field may be $\Q$ or quadratic. In this case, the Rosati involution is trivial. This gives case (1)-(a,b). By \cite{MumAV}*{p.~196}, the Rosati involution of Type II is the transpose and its invariants are symmetric 2-by-2 matrices, which proves case (1)-(c). In the case of Type IV, $D$ is a degree $4$ CM field. In this case, the Rosati involution is the complex conjugation and this gives case (1)-(d). When $A$ is not geometrically simple, then $A$ is isogenous to the product of two elliptic curves and all these cases are easy.
 
Notice that after a suitable field extension, there exists a density one set of primes such that $A$ has ordinary reduction (due to Katz, see \cite{Ogus} Sec. 2). We first pass to such an extension and only focus on primes where $A$ has ordinary reduction. Then $r_\frakp=2$ if $A$ mod $\frakp$ is geometrically simple and $r_\frakp=4$ if $A$ is not. Since $r_\frakp\geq r$, we see that $r_\frakp=4$ in (1)-(c), (2)-(a,b) for any $\frakp$ where $A$ has ordinary reduction. 
When $r=2$ (case (1)-(b,d), (2)-(c)), the dimension over $\Q$ of the orthogonal complement $T$ of $\NS(\ol A)$ in the Betti cohomology $\HH^2(A,\Q)$ is $4$.
By \cite{FC}*{Thm.~1}, if $r_\frakp$ were $4$ for a density one set of primes, then the endomorphism algebra $E$ of $T$ as a Hodge structure would have been a totally real field of degree $r_\frakp-r=2$ over $\Q$. Then $T$ would have been of dimension $2$ over $E$, which contradicts the assumption of the second part of Charles' theorem.
Now the remaining case is (1)-(a).
By \cite{FC}, for a density one set of $\frakp$, the rank $r_\frakp$ only depends on the degree of the endomorphism algebra $E$ of the transcendental part $T$ of the $\HH^2(A, \mathbb Q)$. This degree is the same for all $A$ in case (1)-(a) since $E= \End(T)\subset \End(\HH^2(A,\Q))$ is a set of Hodge cycles of $A\times A$ and all $A$ in this case have the same set of Hodge cycles. For more details we refer the reader to \cite{Can16}.
Hence we only need to study a generic abelian surface. For a generic abelian surface, its ordinary reduction is a (geometrically) simple CM abelian surface and hence $r_\frakp$ is $2$.
\end{proof}
 
\subsubsection{Algorithms to compute the geometric N\'eron--Severi rank of $A$}\label{algorithmsNSrank}
 
Here we discuss an algorithm provided by Charles in \cite{FC}. Charles' algorithm is to compute the geometric N\'eron--Severi rank of any $K3$ surface $X$, and his algorithm relies on the Hodge conjecture for codimension $2$ cycles in $X\times X$. However, the situation where the Hodge conjecture is needed does not occur for abelian surfaces, so his algorithm is unconditional for abelian surfaces.

Suppose that $A$ is a principally polarized abelian surface and $\Theta$ its principal polarization. We run the following algorithms simultaneously:
\begin{enumerate}
\item Compute Hilbert schemes of curves on $A$ with respect to $\Theta$ for each Hilbert polynomial, and find divisors on $A$. Compute its intersection matrix using the intersection theory, and determine the rank of lattices generated by divisors one finds. This gives a lower bound $\rho$ for $r = \rk\NS(\overline{A})$.
\item\label{Artin--Tate} For each finite place $\frakp$ of good reduction for $A$, compute the geometric N\'eron--Severi rank $r_\frakp$ for $A_{\frakp}$ using explicit point counting on the curve $C$ combined with the Weil conjecture and the Tate conjecture. Furthermore compute the square class $\delta(\frakp)$ of the discriminant of $\NS(A_{\frakp})$ in $\mathbb Q^\times / (\mathbb Q^\times)^2$ using the Artin--Tate conjecture: 
\begin{displaymath}
P_2(q^{-s}) \sim_{s \rightarrow 1} \left(\frac{\#\Br(A_{\frakp}) \cdot |\Disc (\NS(A_{\frakp}))|}{q} (1-q^{1-s})^{\rho(A_\frakp)}\right),
\end{displaymath}
where $P_2$ is the characteristic polynomial of the Frobenius endomorphism on 
\begin{displaymath}
\HH^2_{\et}(\overline{A}_\frakp, \mathbb Q_\ell),
\end{displaymath}
and $q$ is the size of the residue field of $\frakp$.
When the characteristic is not equal to $2$, then the Artin-Tate conjecture follows from the Tate conjecture for divisors (\cite{Mil75}), and the Tate conjecture for divisors in abelian varieties is known (\cite{Tate66}).
Note that as a result of \cite{LLR05}, the size of the Brauer group must be a square.
This gives us an upper bound for $r$.
\end{enumerate}

When $r$ is even, there exists a prime $\frakp$ such that $r=r_\frakp$. Thus eventually we obtain $r_\frakp = \rho$ and we compute $r$.

When $r$ is odd, it is proved in \cite{FC}*{Prop.~18} that there exist $\frakp$ and $\mathfrak q$ such that $r_\frakp = r_{\mathfrak q} = \rho+1$, but $\delta(\frakp) \neq \delta(\mathfrak q)$ in $\mathbb Q^\times / (\mathbb Q^\times)^2$. If this happens, then we can conclude that $r= r_\frakp -1$.

\begin{remark}
The algorithm (1) can be conducted explicitly in the following way: Suppose that our curve $C$ of genus $2$ is given as a subscheme in the weighted projective space $\mathbb P(1, 1, 3)$. Let $Y = \Sym^2(C)$ be the symmetric product of $C$. Then we have the following morphism
\begin{displaymath}
f : C \times C \rightarrow Y \rightarrow \Jac(C), \quad (P, Q) \mapsto [P+Q - K_C].
\end{displaymath}
The first morphism $C\times C \rightarrow Y$ is the quotient map of degree $2$, and the second morphism is a birational morphism contracting a smooth rational curve $R$ over the identify of $\Jac(C)$. We denote the diagonal of $C\times C$ by $\Delta$ and the image of the morphism $C \ni P \mapsto (P, \iota(P)) \in C \times C$ by $\Delta'$ where $\iota$ is the involution associated to the degree $2$ canonical linear system. Then we have
\begin{displaymath}
f^*\Theta \equiv 5p_1^*\{\mathrm{pt}\} + 5p_2^*\{\mathrm{pt}\} - \Delta.
\end{displaymath}
Note that $f^*\Theta$ is big and nef, but not ample.
If we have a curve $D$ on $\Jac(C)$, then its pullback $f^*D$ is a connected subscheme of $C\times C$ which is invariant under the symmetric involution and $f^*D . \Delta' = 0$, and vice verse. Hence instead of doing computations on $\Jac(C)$, we can do computations of Hilbert schemes and the intersection theory on $C\times C$. This may be a more effective way to find curves on $\Jac(C)$ and its intersection matrix.
\end{remark}

\begin{remark}
The algorithm (2) is implemented in the paper \cite{EJ}. 
\end{remark}

\subsection{the computation of the N\'eron--Severi lattice and its Galois action}
 
Here we discuss an algorithm to compute the N\'eron--Severi lattice and its Galois structure. We have an algorithm to compute the N\'{e}ron--Severi rank of $\overline{A}$, so we may assume it to be given. First we record the following algorithm:
\begin{lemma}
\label{lemma: overlattice}
Let $S$ be a polarized abelian surface or a polarized K3 surface over $k$, with an ample divisor $H$. Suppose that we have computed a full rank sublattice $M \subset \NS(\overline{S})$ containing the class of $H$, i.e., we know its intersection matrix, the Galois structure on $M\otimes \mathbb Q$, and we know generators for $M$ as divisors in $S$.
Then there is an algorithm to compute $\NS(\overline{S})$ as a Galois module.
\end{lemma}
\begin{proof}
We fix a basis $B_1, \cdots, B_r$ for $M$ which are divisors on $S$.
First note that the N\'eron--Severi lattice $\NS(\overline{S})$ is an overlattice of $M$. By Nikulin \cite{Nik}*{Sec.~1-4}, there are only finitely many overlattices, (they correspond to isotropic subgroups in $D(M) = M^\vee/M$), and moreover we can compute all possible overlattices of $M$ explicitly. Let $N$ be an overlattice of $M$. We can determine whether $N$ is contained in $\NS(\overline{S})$ in the following way:
 
Let $D_1, \cdots, D_s$ be generators for $N/M$. The overlattice $N$ is contained in $\NS(\overline{S})$  if and only if the classes $D_i$ are represented by integral divisors. After replacing $D_i$ by $D_i + mH$, we may assume that $D_i^2 >0$ and $\inter{D_i}{H} >0$. If $D_i$ is represented by an integral divisor, then it follows from Riemann--Roch that $D_i$ is actually represented by an effective divisor $C_i$. We define $k = \inter{D_i}{H}$ and $c= -\frac{1}{2}D_i^2$. 
The Hilbert polynomial of $C_i$ with respect to $H$ is $P_i(t)=kt +c$. Now we compute the Hilbert scheme $\mathrm{Hilb}^{P_i}$ associated with $P_i(t)$. For each connected component of $\mathrm{Hilb}^{P_i}$, we take a member $E_i$ of the universal family and compute the intersection numbers $\inter{B_1}{E},\ldots,\inter{B_r}{E}$. If these coincide with the intersection numbers of $D_i$, then that member $E_i$ is an integral effective divisor representing $D_i$. If we cannot find such an integral effective divisor for any connected component of $\mathrm{Hilb}^{P_i}$, then we conclude that $N$ is not contained in $\NS(\overline{S})$.
 
In this way we can compute the maximal overlattice $N_{\max}$ all whose classes are represented by integral divisors. This lattice $N_{\max}$ must be $\NS(\overline{S})$. Since $M$ is full rank, the Galois structure on $M$ induces the Galois structure on $\NS(\overline{S})$.
\end{proof}
 
\subsubsection{$\rk \, \NS(\overline{A}) = 1$}\label{rank1}
 
The goal of this subsection is to prove the following proposition:
 
\begin{prop}
Let $A$ be a principally polarized abelian surface defined over a number field $k$ whose
geometric N\'{e}ron--Severi rank is $1$. Let $X$ be the Kummer surface associated to $A$.
Then there is an explicit algorithm that computes $\NS(\overline{X})$ as a Galois module and furthermore computes the group $\Br_1(X)/\Br_0(X)$.
\end{prop}

The abelian surface $A$ is a principally polarized abelian surface, so the lattice $\NS(\overline{A})$ is isomorphic to the lattice $\langle 2\rangle$ with the trivial Galois action.
We denote the blow up of 16 $2$-torsion points on $A$ by $\widetilde{A}$ and the $16$ exceptional curves on $\tilde{A}$ by $E_i$. There is an isometry
\begin{displaymath}
\NS(\widetilde{A}_{\ol{k}} ) \cong \NS(\overline{A}) \oplus \bigoplus_{i=1}^{16} \mathbb Z E_i.
\end{displaymath}
We want to determine the Galois structure of this lattice. To this end, one needs to understand the Galois action on the set of $2$-torsion elements on $\ol A$. This can be done explicitly in the following way: Suppose that $A$ is given as a Jacobian of a smooth projective curve $C$ of genus $2$. Then $C$ is a hyperelliptic curve whose canonical linear series is a degree $2$ morphism. We denote the ramification points (over $\ol{k}$) of this degree $2$ map by $p_1, \cdots, p_6$. One can find the Galois action on these ramification points from the polynomial defining $C$. All non-trivial $2$-torsion points of $\ol A$ are given by $p_i-p_j$ where $i<j$. Note that $p_i - p_j \sim p_j-p_i$ as classes in $\Pic(C)$. Thus, we can determine the Galois structure on the set of $2$-torsion elements of $\ol A$.

Let $X$ be the Kummer surface associated to $A$ with the degree 2 finite morphism $\pi : \widetilde{A} \rightarrow X$. We take the pushforward of $\NS(\widetilde{A}_{\bar{k}})$ in $\NS(\overline{X} )$:
\begin{displaymath}
\NS(\overline{X}) \supset \pi_*\NS(\widetilde{A}_{\bar{k}}) \cong \pi_*\NS(\overline{A}) \oplus \bigoplus_{i=1}^{16} \mathbb Z \pi_*E_i.
\end{displaymath}
This is a full rank sublattice.
Thus the Galois representation for $\NS(\widetilde{A}_{\bar{k}})$ tells us the representation for $\NS(\overline{X})$. Hence we need to determine the lattice structure for $\NS(\overline{X})$. This is done in \cite{LP80}*{Sec.~3}. Let us recall the description of the N\'eron--Severi lattice for any Kummer surface.

According to \cite{LP80}*{Prop.~3.4} and \cite{LP80}*{Prop.~3.5}, the sublattice $\pi_*\NS(\widetilde{A}_{\bar{k}})$ is primitive in $\NS(\overline{X})$, and its intersection pairing is twice the intersection pairing of $\NS(\widetilde{A}_{\bar{k}})$. In particular, in our situation, we have $\pi_*\NS(\widetilde{A}_{\bar{k}}) \cong \langle 4\rangle$. Let $K$ be the saturation of the sublattice generated by the $\pi_*E_i$'s. Nodal classes $\pi_*E_i$ have self intersection $-2$. We have the following inclusions:
\begin{displaymath}
\bigoplus_{i=1}^{16} \mathbb Z \pi_*E_i \subset K \subset K^{\vee} \subset \left(\bigoplus_{i=1}^{16} \mathbb Z \pi_*E_i\right)^\vee = \bigoplus_{i=1}^{16} \frac{1}{2}\mathbb Z \pi_*E_i
\end{displaymath}
where $L^\vee$ denotes the dual abelian group of a given lattice $L$. We denote the set of $2$-torsion elements of $\ol A$ by $V$. We can consider $V$ as the $4$ dimensional affine space over $\mathbb F_2$. Then we can interpret  $\bigoplus_{i=1}^{16} \frac{1}{2}\mathbb Z\pi_*E_i/\mathbb Z\pi_*E_i$ as the space of $\frac{1}{2}\mathbb Z/\mathbb Z$-valued functions on $V$. \cite{LP80}*{Prop~3.6} shows that with this identification, the image of $K$ (resp. $K^\vee$) in $\bigoplus_{i=1}^{16} \frac{1}{2}\mathbb Z/\mathbb Z$ consists of polynomial functions $V \rightarrow \frac{1}{2}\mathbb Z/\mathbb Z$ of degree $\leq 1$ (resp. $\leq 2$.) Hence we have
\begin{displaymath}
\left[K : \bigoplus_{i=1}^{16} \mathbb Z \pi_*E_i\right] = 2^5, \quad [K^\vee : K ]= 2^6.
\end{displaymath}
This description allows us to choose an explicit basis for $K$ as well as to find its intersection matrix. The discriminant group of $K$ is isomorphic to $\mathbb F_2^6$ whose discriminant form is given by
\begin{displaymath}
\begin{pmatrix}
0&0&0&0&0&\frac{1}{2}\\
0&0&0&0&\frac{1}{2}&0\\
0&0&0&\frac{1}{2}&0&0\\
0&0&\frac{1}{2}&0&0&0\\
0&\frac{1}{2}&0&0&0&0\\
\frac{1}{2}&0&0&0&0&0
\end{pmatrix}.
\end{displaymath}
This discriminant form is isometric to the discriminant form of $\pi_*H^2(A, \mathbb Z)$ which is isomorphic to
\begin{displaymath}
\begin{pmatrix}
0&2\\
2&0
\end{pmatrix}
\oplus
\begin{pmatrix}
0&2\\
2&0
\end{pmatrix}
\oplus
\begin{pmatrix}
0&2\\
2&0
\end{pmatrix}
\end{displaymath}
Now we have overlattices:
\begin{displaymath}
\pi_*\NS(\overline{A}) \oplus K \subset \NS(\overline{X}).
\end{displaymath}
To identify $\NS(\overline{X})$, we consider the following overlattices:
\begin{displaymath}
\pi_* \HH^2(A, \mathbb Z) \oplus K \subset \HH^2(X, \mathbb Z).
\end{displaymath}
One can describe $\HH^2(X, \mathbb Z)$ using techniques in \cite{Nik}*{Sec~1.1-1.5}. Since the second cohomology of any $K3$ surface is unimodular, we have the following inclusions:
\begin{displaymath}
\pi_* \HH^2(A, \mathbb Z) \oplus K \subset \HH^2(X, \mathbb Z) = \HH^2(X, \mathbb Z)^\vee \subset (\pi_* \HH^2(A, \mathbb Z))^\vee \oplus K^\vee
\end{displaymath}
This gives us the following isotropic subgroup in the direct sum of the discriminant forms:
\begin{displaymath}
H = \HH^2(X, \mathbb Z)/\pi_* \HH^2(A, \mathbb Z) \oplus K \hookrightarrow D(\pi_* \HH^2(A, \mathbb Z)) \oplus D(K)
\end{displaymath}
where $D(L)$ denotes the discriminant group of a given lattice $L$. 

Since $\pi_* \HH^2(A, \mathbb Z)$ and $K$ are both primitive in $\HH^2(X, \mathbb Z)$, each projection $H \rightarrow D(\pi_* \HH^2(A, \mathbb Z))$ and $H \rightarrow D(K)$ is injective. Moreover, since $\HH^2(X, \mathbb Z)$ is unimodular, the isotropic subgroup $H$ must be maximal inside $D(\pi_* \HH^2(A, \mathbb Z)) \oplus D(K)$. This implies that both injections are in fact isomorphisms. Thus we determine $\HH^2(X, \mathbb Z)$ as an overlattice corresponding to $H$ in $D(\pi_* \HH^2(A, \mathbb Z)) \oplus D(K)$. Note that we can apply the orthogonal group $O(K)$ to $H$ so that $H$ is unique up to this action. Namely if we fix an identification $q_K = -q_K \cong q_{\pi_* \HH^2(A, \mathbb Z)}$ and $D(K) \cong D(\pi_* \HH^2(A, \mathbb Z))$, then we can think of $H$ as the diagonal in $D(K) \oplus D(\pi_* \HH^2(A, \mathbb Z))$.

We succeeded in expressing our embedding $\pi_* H^2(A, \mathbb Z) \oplus K \hookrightarrow H^2(X, \mathbb Z)$, hence we can express $\NS(\overline{X})$ as
\begin{displaymath}
\NS(\overline{X}) = H^2(X, \mathbb Z) \cap (\pi_*\NS(\overline{A}) \oplus K)\otimes \mathbb Q.
\end{displaymath}
Note that an embedding of $\NS(\overline{A})$ into $H^2(A, \mathbb Z)$ is unique up to isometries because of \cite{Nik}*{Thm~1.1.2\footnote{attributed to D.G. James}}, so we can map a generator of $\NS(\overline{A})$ to $e+f$ where $e, f$ is a basis for the hyperbolic plane $U = \begin{pmatrix}0&1\\ 1&0 \end{pmatrix}$. Thus we determine the lattice structure of $\NS(\overline{X})$.

\begin{remark}
In Section \ref{computations}, we will in fact use a somewhat simpler argument in order to describe $\NS(\ol X)$ as a Galois module. The advantage of the argument given in the current section is that it can be made applicable for higher rank cases. 
\end{remark}

\subsubsection{$\rk \, \NS(\overline{A}) = 2$}\label{rankNSA2}
 
The goal of this subsection is to prove the following proposition:
 
\begin{prop}
\label{prop:rank2}
Let $A$ be a principally polarized abelian surface defined over a number field $k$ whose geometric N\'{e}ron--Severi rank is $2$. Let $X$ be the Kummer surface associated to $A$. Then there is an explicit algorithm that computes $\NS(\overline{X})$ as a Galois module and furthermore computes the group $\Br_1(X)/\Br_0(X)$.
\end{prop}
 
Once we compute $\NS(\overline{A})$, then we can compute $\NS(\overline{X})$ using techniques in the previous section, so here we focus on the computation of $\NS(\overline{A})$.
 
Using the Artin--Tate conjecture (see \ref{Artin--Tate} in section \ref{algorithmsNSrank}), one can compute a positive square free integer $d$ satisfying
\begin{displaymath}
\Disc(\NS(\overline{A})) \equiv -d \text{ in } \mathbb Q^\times/(\mathbb Q^\times)^2
\end{displaymath}
where we  may use any prime $p$ such that the N\'{e}ron--Severi rank of the reduction mod $p$ is $2$ since the discriminants $\Disc(\NS(\ol A))$ and $\Disc(\NS(\ol{A_p}))$ are equal in such a situation.
 
Suppose $d = 1$. In this case, since $\Disc(\NS(\overline{A}))$ is a negative square, there exists a primitive class $D \in \NS(\overline{A})$ such that $\inter{D}{D} = 0$ and $\inter{D}{\Theta} >0$. This implies that $A$ contains an elliptic subgroup $E_1$. Let $\inter{\Theta}{E_1} = n$.
Then $\NS(\overline{A})$ contains the rank $2$ lattice generated by $\Theta$ and $E_1$ as a full rank sublattice whose intersection matrix is given by 
\begin{displaymath}
M=
\begin{pmatrix}
2 & n \\
n & 0
\end{pmatrix}.
\end{displaymath}
The nef cone of A is the positive cone of this intersection matrix given by
\begin{displaymath}
2x^2 + 2nxy \geq 0.
\end{displaymath}
This cone has two rational extremal rays, one is generated by $E_1$ and another is generated by another elliptic subgroup $E_2$ which has the form of $\frac{1}{a}(n \Theta -E_1)$ for some positive integer $a$. Then $\inter{\Theta}{E_2} = n/a$ must be an integer, so $a$ divides $n$. We have $E_1= n\Theta-aE_2$, but the fact that $E_1$ has to be primitive implies that $a=1$.

Therefore starting from $e=1$, we search a for curve $D$ such that $\inter{\Theta}{D} = e$ and $D^2=0$ using Hilbert scheme computations. The first such $e$ where we find a such curve equals $n$.

To identify $\NS(\overline{A})$, which is an overlattice of $M$, we compute the discriminant form of $M$.

When $n$ is odd, the discriminant group $D(M)$ is isomorphic to $\mathbb Z/n^{2}\mathbb Z$ whose generator takes the value $-\frac{2}{n^2}$. Any isotropic subgroup has the form of $I_l := ln\mathbb Z/n^{2}\mathbb Z$ where we may choose $l$ dividing $n$. Let $D_l$ be a divisor class generating $I_l$. Any candidate for this divisor satisfies $\inter{\Theta}{D_l}=0$ and $\inter{E_1}{D_l}= 2nl$. Then $D_l + n \Theta$ is an effective divisor satisfying the intersections $\inter{\Theta}{D_l + n\Theta}= 2n$ and $\inter{E_1}{D_l +n \Theta}= n(2l+n)$. Thus for each $e$ dividing $n$ we search for divisors with these intersection properties, and the smallest $e$ for which we can find such a divisor equals $l$. Therefore, we can determine all isotropic subgroups of $D(M)$ and hence all overlattices of $M$.

When $n$ is even, the discriminant group $D(M)$ is isomorphic to $\mathbb Z/2\mathbb Z \oplus \mathbb Z/(n^2/2)\mathbb Z$ whose discriminant form takes values $q(1,0)=\frac{1}{2}$ and $q(0,1)=-\frac{2}{n^2}$. Again there are only finitely many isotropic subgroups and for each subgroup, we can test whether generators of subgroups are integral divisor classes or not by adding a multiple of the theta divisor and doing the Hilbert schemes computation. Thus we can determine $\NS(\overline{A})$ as in the odd $n$ case. 

Now we would like to understand the Galois structure. To this end, we compute the Hilbert scheme of curves that have the Hilbert polynomial $nt$. This Hilbert scheme parametrizes elliptic curves in $A$, and over an algebraically closed field it is the disjoint union of two elliptic curves. If they are conjugate to each other, then the Galois action swaps $E_1$ and $E_2$. Otherwise the Galois action on $\NS(\ol A)$ is trivial.

Suppose $d > 1$. In this case, $A$ is geometrically simple. The endomorphism algebra $\End^0_{\bar{k}}(A)$ is either a totally real quadratic field or a degree $4$ CM field over a totally real quadratic field. According to \cite{MumAV}*{p.~208}, we have an isomorphism
\begin{displaymath}
\NS(\overline{A}) \cong \End_{\bar{k}}(A)^\dagger
\end{displaymath}
where $\dagger$ is the Rosati involution with respect to $\Theta$. When $\End^0_{\bar{k}}(A)$ is a totally real quadratic field, the Rosati involution is trivial. If $\End^0_{\bar{k}}(A)$ is a CM field, then $\dagger$ is the complex conjugation. In either case, $\End_{\bar{k}}(A)^\dagger$ is an order in a totally real quadratic field $\End^0_{\bar{k}}(A)^\dagger$ which is isomorphic to $\mathbb Q(\sqrt{d})$. We define $\omega \in \mathbb Q(\sqrt{d})$ by
\begin{displaymath}
\omega = 
\begin{cases}
1+\sqrt{d} & \text{ if } d \equiv 2, 3 \bmod 4,\\
\frac{1+\sqrt{d}}{2}& \text{ if } d \equiv 1 \bmod 4.
\end{cases}
\end{displaymath}
Then every order $\mathcal O$ in $\mathbb Q(\sqrt{d})$ has the form of
\begin{displaymath}
\mathcal O = \mathbb Z \oplus \mathbb Z f\omega,
\end{displaymath}
for some positive integer $f \in \mathbb Z$. 

To find $f$, consider the curve class $D_e := (e\omega)^*\Theta$. By \cite{MumAV}*{p.~192, Thm.~1} this class satisfies
\begin{displaymath}
\inter{D_e}{\Theta} = \Tr_{\mathbb Q(\sqrt{d})/\mathbb Q}(e^2\omega^2), \quad D_e^2 = 2e^4\Nm_{\mathbb Q(\sqrt{d})/\mathbb Q}(\omega)^2.
\end{displaymath}
Starting from $e=1$, we search for curves with these intersection numbers using Hilbert scheme computations, and the first $e$ where we find such a curve is equal to $f$. The Galois action on $D_f$ determines the Galois structure for $\NS(\overline{A})$.

To compute $\NS(\overline{X})$ where $X$ is the Kummer surface associated to $A$, one needs to specify an embedding of $\NS(\overline{A})$ into $\HH^2(A, \mathbb Z)$. It follows from \cite{Nik}*{Thm~1.14.4} that such an embedding is unique up to automorphisms, so we only need to specify one embedding.

Choose a basis $D_1, D_2$ for $\NS(\overline{A})$ such that we may set $D_1^2 = 2a >0$, $\inter{D_1}{D_2} = b$, and $D_2^2 = 2c$.
We have an isomorphism
\begin{displaymath}
\HH^2(A, \mathbb Z) \cong \begin{pmatrix} 0&1\\1&0\end{pmatrix}\oplus\begin{pmatrix} 0&1\\1&0\end{pmatrix}\oplus \begin{pmatrix} 0&1\\1&0\end{pmatrix}.
\end{displaymath}
We denote the basis by $e_1, f_1, e_2, f_2, e_3, f_3$. The desired embedding is given by
\begin{displaymath}
D_1 \mapsto e_1 + af_1, \quad D_2 \mapsto bf_1 + e_2 + cf_2.
\end{displaymath}
Thus Proposition~\ref{prop:rank2} follows. Alternatively we can realize $Y = A/\pm1$ as a quartic surface in $\mathbb P^3$ (\cite{FS97}). We may apply the Hilbert scheme computations to this surface $Y$.

\subsubsection{$\rk \, \NS(\overline{A}) = 3$}

\begin{prop}
\label{prop: rank3}
Let $A$ be a principally polarized abelian surface defined over a number field $k$ whose geometric N\'{e}ron--Severi rank is $3$. Let $X$ be the Kummer surface associated to $A$. Then there is an explicit algorithm that computes a bound on the size of $\Br_1(X)/\Br_0(X)$.
\end{prop}

In this case two possible situations arise: the first one when $A$ is simple and $D$ is an indefinite quaternion algebra, and the second one when $A$ is isogenous to the product $E \times E$ where $E$ is an elliptic curve without complex multiplication. Distinguishing between the two case is possible by using the Hasse principle, i.e., whether a sublattice of $\NS(\overline{A})$ contains a non-trivial vector $D$ representing zero, i.e., $D^2=0$.

The second case can be handled in a way similar to the first case, so we focus on the case when $A$ is simple with a quaternion multiplication.

One can compute $\NS(\overline{A})$ using Lemma~\ref{lemma: overlattice}. To apply that lemma, we need a full rank sublattice. Such a lattice is a byproduct of the algorithm discussed in section \ref{algorithmsNSrank} that one may have used to determine that the N\'{e}ron--Severi rank is 3. Since $\NS(\overline{X})$ is an overlattice of $\pi^*\NS(\overline{A}) \oplus K$, there are only finitely many possibilities for $\NS(\ol X)$. Thus our assertion follows. Again alternatively we apply the Hilbert scheme computations, intersection numbers computations, and Lemma~\ref{lemma: overlattice} to a minimal resolution $X$ of the quartic surface. This provides us an exact computation of $\NS(\overline{X})$.

\begin{remark}
One can create a list of finitely many candidates for $\NS(\overline{A})$. The endomorphism algebra $D$ is a quaternion algebra given by the Hilbert symbol $(a,b)$ where $a, b$ are square free integers. (Then $\NS(\overline{A})_{\mathbb Q}$ is given by a quadratic form proportional to $z^2 - ax^2 -by^2$, so once we know what $\NS(\overline{A})$ is, we can compute $D$.) An example of maximal orders for a quaternion algebra $D$ is given in \cite{albert}, and for any fixed maximal order of $D$ there is an abelian surface $A'$ isogenous to $A$ such that $A'$ has multiplication by the given maximal order. To see this, one first notices that all maximal orders in the indefinite quaternion algebra $D$ over $\Q$ are conjugate (see \cite{Clark}*{Chp.~0, Cor.~39}) and so we may assume that $\End(A)$ is contained in the maximal order that we have chosen. Then by \cite{GR14}*{Thm.~1.2}, there is an isogeny between $A$ and $A'$ whose degree can be bounded by a constant in terms of $h(A)$ and $[k:\Q]$. For such $A'$, one can compute its N\'eron--Severi lattice by using \cite{DR04}*{Thm.~3.1}.
\end{remark}

\subsubsection{$\rk \, \NS(\overline{A}) = 4$}

In this section, we discuss the following proposition:

\begin{prop}
\label{prop: rank4}
Let $A$ be a principally polarized abelian surface defined over a number field $k$ whose geometric N\'{e}ron--Severi rank is $4$. Let $X$ be the Kummer surface associated to $A$. Then there is an explicit algorithm that computes a bound on the size of $\Br_1(X)/\Br_0(X)$.
\end{prop}
The proof of Proposition~\ref{prop: rank3} applies to this case.
Moreover, one can create a finite list of possibilities for $\Br_1(X)/\Br_0(X)$. Indeed,
in this case, the abelian surface $A$ is isogenous to a self product $E\times E$ of an elliptic curve $E$ with complex multiplication. 
Moreover, one can assume that the endomorphism ring $\End_{\bar{k}}(E)$ is the maximal order of an imaginary quadratic field $\mathbb Q(\sqrt{-d})$ where $d$ is a positive square free integer. The degree $e$ of an isogeny from $A$ to $E\times E$ is bounded by an explicit constant depending on the Faltings' height by Lemma~\ref{lemma: isogenyI} and Theorem~\ref{theorem: isogenyI}. According to \cite{Kani}*{Cor.~23}, the square class of $\Disc(\NS(\ol A))$ is
\begin{displaymath}
\Disc(\NS(\overline{A})) \equiv -d  \mod \mathbb (\mathbb Q^\times)^2,
\end{displaymath}
so the value $d$ can be determined by using the Artin-Tate conjecture. Then the lattice structure of $\NS((E\times E)_{\bar{k}})$ is given by \cite{Kani}*{Prop.~22}, so $\NS(\overline{A})$ is an overlattice of the lattice defined by $e$ times the intersection matrix of $\NS((E\times E)_{\bar{k}})$. Thus there are only finitely many possibilities for $\NS(\overline{A})$. Then $\NS(\overline{X})$ is an overlattice of $\pi^*\NS(\overline{A}) \oplus K$.
Again alternatively we apply Hilbert scheme computations to the quartic surface $Y = A/\pm 1$.

\section{Effective bounds for the transcendental part of Brauer groups}\label{Trans_part}
\subsection{The general case}
\label{subsec: general case}
Let $A$ be a principally polarized abelian surface defined over a number field $k$. Let $X = \Kum(A)$ be the Kummer surface associated to the abelian surface $A$. The goal of this section is to prove the following theorem:

\begin{theorem}
\label{theorem: effectivebounds}
There exists an effectively computable constant $M_1$ depending on the number field $k$, the Faltings' height $h(A)$, and $\NS(\ol{A})$ satisfying
\begin{displaymath}
\#\frac{\Br(X)}{\Br_1(X)} \leq M_1.
\end{displaymath}
\end{theorem}

\begin{remark}\label{rmk_bc}
After taking an algebraic extension $k\subset k'$, the image of $\Br(X)$ in $\Br(\Xbar)$ is contained in the image of $\Br(X_{k'})$, so in order to find a bound for the transcendental part as desired, taking algebraic field extensions is allowed. 
\end{remark}

First we use the following important theorem by Skorobogatov and Zarhin:
\begin{theorem}{\cite{SZ12}*{Prop.~1.3}}\label{thm_SZ1.3}
Let $A$ be an abelian surface defined over a number field $k$ and $X = \Kum(A)$ the associated Kummer surface. Then there is a natural map
\begin{displaymath}
\Br(\ol{X}) \cong \Br(\ol{A})
\end{displaymath}
which is an isomorphism of Galois modules.
\end{theorem}
Hence there is an injection
\begin{displaymath}
\frac{\Br(X)}{\Br_1(X)} \hookrightarrow \Br(\ol{X})^{\Gamma} =
\Br(\ol{A})^\Gamma,
\end{displaymath}
where $\Gamma= \Gal(\bar{k}/k)$. Thus, to bound $\frac{\Br(X)}{\Br_1(X)}$ in terms of $k$, the Faltings' height $h(A)$, and  $\delta = \det (\NS(\ol{A}))$, we only need to bound $\Br(\bar{A})^\Gamma$.

Also we would like to recall the following important result about the geometric Brauer groups:

\begin{theorem}
\label{theorem: geometricBrauer}
As abelian groups, we have the following isomorphisms:
\begin{displaymath}
\Br(\ol{X}) \cong \Br(\ol{A}) \cong (\mathbb Q/\mathbb Z)^{6 -\rho},
\end{displaymath}
where $\rho = \rho(\ol{A})$ is the geometric N\'{e}ron--Severi rank of $A$.
\end{theorem}
\begin{proof}
This follows from the remark before \cite{SZ12}*{Lem.~1.1}.
\end{proof}

We discuss several lemmas to prove our main Theorem ~\ref{theorem: effectivebounds}:

\begin{lemma}
There exists an effectively computable integer $M_2$ depending on $k$, $h(A)$, and $\delta = \det (\NS(\ol{A}))$ such that for any prime number $\ell > M_2$ we have
\begin{displaymath}
\Br(\ol{A})^\Gamma_\ell = \{0\},
\end{displaymath}
where $\Br(\ol{A})^\Gamma_\ell$ denotes the $\ell$-torsion group of $\Br(\ol{A})^\Gamma$.
\end{lemma}
\begin{proof}
This essentially follows from results in \cite{SZ08} combined with Theorem~\ref{effFal}. The following exact sequence occurs as the $n=1$ case of \cite{SZ08}*{p.~486 (5)}:
\begin{align*}
0 &\to \left(\NS(\ol{A})/\ell\right)^\Gamma \stackrel{f}{\to} \HH^2_\et(\ol{A},\mu_{\ell})^\Gamma \to \Br(\ol{A})_{\ell}^\Gamma\to \\
&\to \HH^1(\Gamma, \NS(\ol{A})/\ell) \stackrel{g}{\to} \HH^1(\Gamma,\HH^2_\et(\ol{A},\mu_{\ell})).
\end{align*}
The discussion in \cite{SZ08}*{Prop.~2.5 (a)} shows that $\NS(\ol{A}) \otimes \mathbb Z_\ell$ is a direct summand of $\HH^2_\et(\ol{A}, \mathbb Z_\ell (1))$ for any prime $\ell \nmid \delta$. For such $\ell$, the homomorphism $g$ in the above exact sequence is injective.

Next, Theorem~\ref{effFal} asserts that there exists an effectively computable integer $C>0$ depending on $k$ and $h(A)$ such that for any prime $\ell \nmid C$, we have an isomorphism:
\begin{displaymath}
\End_k(A)/\ell \cong \End_\Gamma(A_\ell, A_\ell).
\end{displaymath}
The discussion in \cite{SZ08}*{Lem.~3.5} shows that for such $\ell$, the homomorphism $f$ is bijective. Thus our assertion follows.
\end{proof}

Thus, to prove our main theorem, we need to bound $\Br(\ol{A})^\Gamma(\ell)$ for each prime number $\ell$ where $\Br(\ol{A})^\Gamma(\ell)$ denotes the $\ell$-primary subgroup of elements whose orders are powers of $\ell$. To achieve this task, we employ techniques from \cite{HKT13} Sections 7 and 8. 

We fix an embedding $k \hookrightarrow \mathbb C$ and consider the following lattice:
\begin{displaymath}
\HH^2(A(\mathbb C), \mathbb Z).
\end{displaymath}
It contains $\NS(\ol{A})$ as a primitive sublattice and we denote its orthogonal complement by $T_A = \langle\NS(\ol{A}) \rangle ^\perp_{\HH^2(A(\mathbb C), \mathbb Z)}$ and call it the transcendental lattice of $A$.
The direct sum $\NS(\ol{A}) \oplus T_A$ is a full rank sublattice of $\HH^2(A(\mathbb C), \mathbb Z)$ and we can put it into the exact sequence:
\begin{displaymath}
0 \rightarrow \NS(\ol{A}) \oplus T_A \rightarrow \HH^2(A(\mathbb C), \mathbb Z) \rightarrow K \rightarrow 0,
\end{displaymath}
where $K$ is a finite abelian group of order $\delta = \det (\NS(\ol{A}))$. Tensoring with $\mathbb Z_\ell$ and using a comparison theorem for the different cohomologies, we have
\begin{displaymath}
0 \rightarrow \NS(\ol{A})_\ell \oplus T_{A,\ell} \rightarrow \HH_{\et}^2(\ol{A}, \mathbb Z_\ell(1)) \rightarrow K _\ell \rightarrow 0,
\end{displaymath}
where $\NS(\ol{A})_\ell = \NS(\ol{A})\otimes \mathbb Z_\ell$, $T_{A,\ell} = T_A \otimes \mathbb Z_\ell$, and $K_\ell$ is the $\ell$-primary part of $K$. The second \'etale cohomology $\HH_{\et}^2(\ol{A}, \mathbb Z_\ell(1))$ comes with a natural pairing which is compatible with $\Gamma$-action, and $T_{S,\ell}$ is the orthogonal complement of $\NS(\ol{A})_\ell$. In particular, $T_{A,\ell}$ has a natural structure as a Galois module.

\begin{lemma}
\label{lemma: transcendental}
For each prime number $\ell$, there is an effectively computable constant $M_3$ depending on $k$, and $h(A)$ such that for each integer $n\geq 1$ the bound
\begin{displaymath}
\# (T_A/\ell^n)^\Gamma \leq \ell^{M_3}
\end{displaymath}
is satisfied.
\end{lemma}

\begin{proof}
Since $A$ is principally polarized, we have a natural isomorphism of Galois modules:
\begin{displaymath}
\HH_\et^1(\ol{A}, \mathbb Z_\ell(1)) \cong (\HH_\et^1(\ol{A}, \mathbb Z_\ell(1)))^* \cong T_\ell(A),
\end{displaymath}
where $T_\ell(A)$ is the Tate module of $A$. Hence we have
\begin{displaymath}
T_{A,\ell} \hookrightarrow \HH_\et^2(\ol{A}, \mathbb Z_\ell(1)) = \largewedge^2 \HH_\et^1(\ol{A}, \mathbb Z_\ell(1)) \hookrightarrow \HH_\et^1(\ol{A}, \mathbb Z_\ell(1)) \otimes \HH_\et^1(\ol{A}, \mathbb Z_\ell(1)) \cong \End(T_\ell(A)).
\end{displaymath}
Thus we have
\begin{displaymath}
(T_A/\ell^n) = (T_{A,\ell}/\ell^n) \hookrightarrow \End(T_\ell(A))/\ell^n = \End(\ol{A}[\ell^n]).
\end{displaymath}
Hence we obtain a homomorphism
\begin{displaymath}
\Phi : (T_A/\ell^n)^\Gamma \hookrightarrow \End_\Gamma (\ol{A}[\ell^n]) \rightarrow \End_\Gamma (\ol{A}[\ell^n])/\End(A).
\end{displaymath}
This composite homomorphism $\Phi$ must be injective because $T_A$ is the transcendental lattice which does not meet the algebraic part $\End(A)$. The order of this quotient is bounded by Theorem~\ref{effFal}.
\end{proof}

Taking a finite extension of $k$ only increases the size of $\Br(\ol{A})^{\Gal(\bar{k}/k')}$, so from now on we assume that the Galois action on the N\'eron--Severi space $\NS(\ol{A})$ is trivial. This is automatically true when the geometric N\'{e}ron--Severi rank of $A$ is $1$.

\begin{lemma}
Suppose that the Galois action on $\NS(\ol{A})$ is trivial.
Then there is an effectively computable constant $M_4$ depending on $k$, $h(A)$, and $\delta$ such that for each prime $\ell$, we have
\begin{displaymath}
\# \Br(\ol{A})^\Gamma(\ell) \leq \ell^{M_4}.
\end{displaymath}
\end{lemma}

\begin{proof}
Recall the exact sequence of \cite{SZ08}*{p.~486 (5)}:
\begin{align*}
\begin{split}
0 &\to \left(\NS(\ol{A})/\ell^n\right)^\Gamma \stackrel{f_n}{\to} \HH^2_\et(\ol{A},\mu_{\ell^n})^\Gamma \to \Br(\ol{A})_{\ell}^\Gamma\to \\
&\to \HH^1(\Gamma, \NS(\ol{A})/\ell^n) \stackrel{g_n}{\to} \HH^1(\Gamma,\HH^2_\et(\ol{A},\mu_{\ell^n})),
\end{split}
\end{align*}
so we need to bound the cokernel of $f_n$ and the kernel of $g_n$ independent of $n$. Note that the kernel of $g_n$ was overlooked in \cite{HKT13}. The group $\HH^1(\Gamma, \NS(\ol{A})/\ell^n)$ is equal to $\Hom(\Gamma, \NS(\ol{A})/\ell^n)$ because by assumption the Galois action is trivial. This group is infinite, so it requires a careful analysis. By Theorem~\ref{theorem: geometricBrauer}, it is enough to bound the orders of elements in $\coker(f_n)$ as well as $\ker(g_n)$ independently of $n$.

Let $\ell^m$ be the order of $K_\ell$ and we assume that $n \geq m$. We have the following exact sequence:
\begin{displaymath}
0 \rightarrow \NS(\ol{A})_\ell \oplus T_{A,\ell} \rightarrow \HH_{\et}^2(\ol{A}, \mathbb Z_\ell(1)) \rightarrow K _\ell \rightarrow 0.
\end{displaymath}
Tensoring by $\Z/\ell^n\Z$ (as $\Z_\ell$-modules) and using $\Tor$ functors, we obtain a four term exact sequence:
\begin{equation}\label{seq2}
0 \rightarrow K_\ell \rightarrow \NS(\ol{A})/\ell^n \oplus T_A/\ell^n \rightarrow \HH_\et^2(\ol{A}, \mu_{\ell^n}) \rightarrow K_\ell \rightarrow 0,
\end{equation}
where we've used that the middle term $\HH_{\et}^2(\ol{A}, \mathbb Z_\ell(1))$ is a free (and hence flat) $\Z_\ell$-module.

Note that the projection
\begin{displaymath}
K_\ell \rightarrow \NS(\ol{A})/\ell^n
\end{displaymath}
is injective because $T_A/\ell^n \rightarrow
\HH^2(\ol{A}, \mu_{\ell^n})$ is injective.
In particular, the Galois action on $K_\ell$ is trivial. We split the exact sequence (\ref{seq2}) as
\begin{displaymath}
0 \rightarrow K_\ell \rightarrow \NS(\ol{A})/\ell^n \oplus T_A/\ell^n \rightarrow C \rightarrow 0,
\end{displaymath}
and
\begin{displaymath}
0 \rightarrow C \rightarrow \HH_\et^2(\ol{A}, \mu_{\ell^n}) \rightarrow K_\ell \rightarrow 0.
\end{displaymath}
These gives us the long exact sequences
\begin{displaymath}
0 \rightarrow K_\ell \rightarrow \NS(\ol{A})/\ell^n \oplus (T_A/\ell^n)^\Gamma \rightarrow C^\Gamma \rightarrow \Hom(\Gamma, K_\ell) \rightarrow \Hom(\Gamma, \NS(\ol{A})/\ell^n)
\oplus \HH^1(\Gamma, T_A/\ell^n),
\end{displaymath}
and
\begin{displaymath}
0 \rightarrow C^\Gamma \rightarrow \HH_\et^2(\ol{A}, \mu_{\ell^n})^\Gamma \rightarrow K_\ell \rightarrow \HH^1(\Gamma, C)\rightarrow \HH^1(\Gamma, \HH_\et^2(\ol{A}, \mu_{\ell^n})).
\end{displaymath}
The map $\Hom(\Gamma, K_\ell) \rightarrow \Hom(\Gamma, \NS(\ol{A})/\ell^n)$ is injective, so the sequence
\begin{displaymath}
0 \rightarrow K_\ell \rightarrow \NS(\ol{A})/\ell^n \oplus (T_A/\ell^n)^\Gamma \rightarrow C^\Gamma \rightarrow 0,
\end{displaymath}
is exact. We conclude that
\begin{displaymath}
\#\coker(f_n) = \frac{\#\HH_\et^2(\ol{A}, \mu_{\ell^n})^\Gamma}{\#\NS(\ol{A})/\ell^n} \leq \frac{\#K_\ell \cdot \$C^\Gamma}{\#\NS(\ol{A})/\ell^n} = \#(T_A/\ell^n)^\Gamma
\end{displaymath}
is bounded independent of $n$ by application of Lemma~\ref{lemma: transcendental}.

Next we discuss a uniform bound on the maximum order of elements in $\ker(g_n)$. The homomorphism $g_n$ is a composition of two homomorphisms:
\begin{displaymath}
\HH^1(\Gamma, \NS(\ol{A})/\ell^n) \rightarrow \HH^1(\Gamma, C) \rightarrow \HH^1(\Gamma, \HH_\et^2(\ol{A}, \mu_{\ell^n})).
\end{displaymath}
The kernel of $\HH^1(\Gamma, C) \rightarrow \HH^1(\Gamma, \HH_\et^2(\ol{A}, \mu_{\ell^n}))$ is bounded by $K_\ell$. We have the exact sequence
\begin{displaymath}
0 \rightarrow \NS(\ol{A})/\ell^n \rightarrow C \rightarrow C/\NS(\ol{A})\rightarrow 0,
\end{displaymath}
which gives the long exact sequence
\begin{displaymath}
0 \rightarrow \NS(\ol{A})/\ell^n \rightarrow C^\Gamma \rightarrow (C/\NS(\ol{A}))^\Gamma \rightarrow \HH^1(\Gamma, \NS(\ol{A})/\ell^n) \rightarrow \HH^1(\Gamma, C).
\end{displaymath}
Thus to finish the proof we need to find an uniform bound for the maximum order of elements in $(C/\NS(\ol{A}))^\Gamma$. To obtain this, we look at the exact sequence
\begin{displaymath}
0 \rightarrow K_\ell \rightarrow T_A/\ell^n \rightarrow C/\NS(\ol{A}) \rightarrow 0.
\end{displaymath}

This gives us the long exact sequence
\begin{displaymath}
0 \rightarrow K_\ell \rightarrow (T_A/\ell^n)^\Gamma \rightarrow (C/\NS(\ol{A}))^\Gamma \rightarrow \Hom(\Gamma, K_\ell).
\end{displaymath}
Note that the group $\Hom(\Gamma, K_\ell)$ is killed by $\#K_\ell$. Finally, $\#(T_A/\ell^n)^\Gamma$ is uniformly bounded by the result of Lemma~\ref{lemma: transcendental}. Therefore the maximum order of elements in $(C/\NS(\ol{A}))^\Gamma$ is uniformly bounded and our assertion follows.
\end{proof}

\subsection{Better bounds for Kummer surfaces from non-simple abelian surfaces}
In this subsection, we always assume that $A$ is not geometrically simple. We may assume that after passing to a finite extension of $k$, the principally polarized abelian surface $A$ is isogenous to $E\times E'$, where $E$ and $E'$ are elliptic curves and furthermore we may assume that they are non-isogenous over $\kbar$ if $E\neq E'$. We use the work of \cite{SZ12} and \cite{New} and results in Section \ref{prodEC} to obtain an upper bound of the transcendental part of $\Br(\Kum(A))$. We use $X$ and $Y$ to denote $\Kum(A)$ and $\Kum(E\times E')$ respectively.  

\begin{lemma}\label{IsogBr}
If $s:E\times E'\rightarrow A$ is an isogeny of degree $d$, then the kernel of the $\Gamma$-invariant map $s^*:\Br(\Abar)\rightarrow \Br(\Ebar\times\Epbar)$ is killed by $d$. In particular,
\begin{displaymath}
\#\frac{\Br(X)}{\Br_1(X)}\leq \#\Br(\Xbar)^\Gamma\leq d^{6-\bar{r}}\# \Br(\Ybar)^\Gamma,
\end{displaymath}
where $\bar{r}$ is the N\'eron--Severi rank of $\Abar$. Moreover, we have
\begin{displaymath}
\#\frac{\Br(X)}{\Br_1(X)}\leq d^{6-\bar{r}}\# \frac{\Br(E\times E')}{\Br_1(E\times E')}.
\end{displaymath}
\end{lemma}
\begin{proof}
The first assertion follows from \cite{ISZ}*{Cor.~1.2}. Moreover, the proof of their corollary shows that the kernel of $s^*$ is of order at most $d^{6-\bar{r}}$. Since $s^*$ is $\Gamma$-invariant, we have the restriction map $s^*:\Br(\Abar)^\Gamma\rightarrow \Br(\Ebar\times \Epbar)^\Gamma$ and the kernel of this restricted map is then of order at most $d^{6-\bar{r}}$. In other words, we have 
\begin{displaymath}
\#\Br(\Abar)^\Gamma\leq d^{6-\bar{r}}\# \Br(\Ebar\times \Epbar)^\Gamma.
\end{displaymath}
By Theorem \ref{thm_SZ1.3}, we have isomorphisms of $\Gamma$-modules
\begin{displaymath}
\Br(\Abar)\cong \Br(\Xbar),\ \Br(\Ebar\times \Epbar)\cong \Br(\Ybar)
\end{displaymath}
and we obtain the first inequality.

To obtain the second inequality, we notice that $s^*:\Br(A)\rightarrow \Br(E\times E')$ induces a map $s^*:\frac{\Br(A)}{\Br_1(A)}\rightarrow \frac{\Br(E\times E')}{\Br_1(E\times E')}$ whose kernel is contained in that of $s^*:\Br(\Abar)\rightarrow \Br(\Ebar\times\Epbar)$. Then
\begin{displaymath}
\#\frac{\Br(A)}{\Br_1(A)}\leq d^{6-\bar{r}}\# \frac{\Br(E\times E')}{\Br_1(E\times E')}.
\end{displaymath}
By \cite{SZ12}*{Thm.~2.4}, we have an injective map 
\begin{displaymath}
\frac{\Br(X)}{\Br_1(X)}\rightarrow\frac{\Br(A)}{\Br_1(A)}
\end{displaymath}
and we obtain the second inequality.
\end{proof}

There are three cases and we will discuss case (1) first and then cases (2) and (3):
\begin{enumerate}
\item $E=E'$ and $E$ has complex multiplication by an order of $\OO_K$;
\item $E=E'$ and $E$ is without complex multiplication;
\item $E$ and $E'$ are not isogenous.
\end{enumerate}
\subsubsection{Case (1)}
The main inputs here are \cite{New}*{Thm.~2.5}\footnote{\cite{New}*{Thm.~2.9} describes the transcendental part of the Brauer group when the endomorphisms of $E$ are not defined over the base field. We will not need this result since to obtain an elliptic curve with complex multiplication by $\OO_K$, we will pass to a field extension where all the endomorphisms are defined.} and Corollary \ref{rk4deg}.

\begin{theorem}\label{Newton}
Let $E''$ be an elliptic curve over a number field $k'$ such that $\End_{k'}(E'')=\OO_K$. Then the following inequality holds:
\begin{displaymath}
\# \frac{\Br(E''\times E'')}{\Br_1(E''\times E'')}\leq \left(25/2\right)^4[k':\Q]^4.
\end{displaymath}
The estimate remains true without the factor $\left(25/2\right)^4$ if $\OO_K^*=\{\pm1\}$ holds.
\end{theorem}
\begin{proof}
This is a direct consequence of \cite{New}*{Thm.~2.5}. Let $\ell$ be a rational prime and let $n(\ell)$ be the largest integer $i$ such that the ring class field $K_{\ell^i}$ of the order $\Z+\ell^i\OO_K$ is contained in $k'$. The non-negative integer $n(\ell)$ is nonzero for only finitely many $\ell$ and we use $t$ to denote the number of nonzero $n(\ell)$. By \cite{New}*{Thm.~2.5 and Prop.~2.2}), we have
\begin{displaymath}
\# \frac{\Br(E''\times E'')}{\Br_1(E''\times E'')}\leq \prod_{\ell}\ell^{2n(\ell)}.
\end{displaymath}
Now we bound the right hand side. Let $h_K$ and $\Delta_K$ be the class number and the discriminant of $K$. Write $m=\prod_{\ell}\ell^{n(\ell)}$ and let $K_m$ be the ring class field of the order $\Z+m\OO_K$. First, by \cite{New}*{Rem.~1, eqn.~2.3}, we have 
\begin{equation}\label{equation:New16.1.2.3}
h_K\prod_{\ell, n(\ell)\geq 1}\left(\ell^{n(\ell)}\left(1-\left(\frac{\Delta_K}{\ell} \right)\frac 1\ell\right)\right)\leq 3[K_m:K],
\end{equation}
where $\left(\frac{\Delta_K} \ell \right)$ is the Legendre symbol, which takes values in $\{0,\pm 1\}$.

Since $K_{\ell^{n(\ell)}}$ is contained in $k'$, then by Lemma \ref{Lem_RCF} we have $[k'K_m:k']\leq 3^{t-1}$ (when $t\geq 1$) and therefore also
\begin{displaymath}
3[K_m:K]\leq 3[k'K_m:K]=(3/2)[k'K_m:\Q]\leq (3/2)\cdot 3^{t-1}[k':\Q].
\end{displaymath}
Using $\prod_{\ell=2,3,5,7}3\ell^{1/2}/(\ell-1)<25$ and the fact that $3\ell^{1/2} < (\ell-1)$ holds for $\ell>7$, we have the intermediate inequality
\begin{equation}\label{equation:4.9}
\prod_{\ell, n(\ell)\geq 1}\ell^{n(\ell)-1/2}\leq 25\prod_{\ell, n(\ell)\geq 1}\left(\frac{\ell^{n(\ell)}}{3}\left(1-\left(\frac{\Delta_K} \ell \right)\frac 1\ell\right)\right)\leq \frac{25\cdot 3[K_m:K]}{3^t}\leq (25/2)[k':\Q],
\end{equation}
where the first inequality follows from the inequality $\left(1-\left(\frac{\Delta_K} \ell \right)\frac 1\ell\right)\geq \left(1-\frac 1\ell\right)=\frac{1}{\ell}(\ell-1)$ that holds for any $\ell$ and specifically $\left(1-\frac 1\ell\right)> 3\ell^{-1/2}$ for $\ell>7$. By $4n(\ell)-2\geq 2n(\ell)$, the desired inequality follows from raising both sides of \eqref{equation:4.9} to the 4th power.

To obtain the last assertion, we notice that when $\OO_K^*=\{\pm1\}$, we have $K_m\subset k'$ by Lemma \ref{Lem_RCF} and we can remove the factor $3$ in inequality \eqref{equation:New16.1.2.3}.
\end{proof}

The following lemma arose from a useful discussion with Hendrik Lenstra.

\begin{lemma}\label{Lem_RCF}
Let $K$ be an imaginary quadratic field with algebraic closure $\ol K$. For any integer $m$ let $K_m$ be the ring class field of $K$ associated to the order $\Z+m\calO_K$.  Then for $m=\prod_p p^{e_p}$ the field $K_m$ contains the compositum $L=\prod_p K_{p^{e_p}}\subset \ol K$ and the extension degree is $[K_m:L]=(\#\mu_K/2)^{\max(0,t-1)}$ where $\mu_K$ is the group of roots of unity in $K$ and $t$ is the number of distinct prime numbers dividing $m$.

\end{lemma}
\begin{proof}
Since for $p|m$ the order $\Z+m\calO_K$ is contained in $\Z+p^{e_p}\calO_K$, it is clear that $K_m$ contains each $K_{p^{e_p}}$ and therefore their compositum $L$.

Let $\mathbb{I}_K$ be the group of id\`{e}les of $K$. By the ``existence theorem'' of global class field theory\cite{NeukirchCFT}*{Thm.~IV.7.1}, there is an order inverting bijection between the set of finite abelian extensions of $K$ and the set of closed subgroups of $\mathbb{I}_K$ of finite index containing $K^\times$, carrying extension degrees to indices. 

Then we have
\begin{align*}
[K_m:K] &= h_K[\widehat{\calO}_m^\times:\mu_K\widehat{\calO}_m^\times]\\
&= h_K[(\calO_K/m\calO_K)^\times:(\Z/m\Z)^\times\cdot \im\mu_K]\\
&= \begin{cases} h_K\cdot\frac{\varphi_K(m)/\varphi(m)}{\#\mu_K/2} & m>1;\\ 
h_K& m=1,\end{cases}
\end{align*}
where for any positive integer $n=\prod_p p^{e_p}$ we denote $\hat{\calO}_n=\prod_p\left(\Z_p+p^{e_p}(\Z_p\otimes \calO_K)\right)$, the image of $\mu_K$ is taken inside $(\calO_K/m\calO_K)^\times$, and $\varphi_K(m)$ denotes the order of $(\calO_K/m\calO_K)^\times$. The map $\Z_{>0}\to \Z_{\geq 0}$ given by $m\mapsto \varphi_K(m)$ is a multiplicative function and is the number field analogue of Eulers totient function $\varphi=\varphi_\Q$.

Since each $K_{p^{e_p}}$ contains the Hilbert class field $K_1$ and each extension $K_{p^{e_p}}/K_1$ is unramified outside $p$, the extension degree $[L:K_1]$ splits as a product 
\begin{displaymath}
[L:K_1] = \prod_{p|m}[K_{p^{e_p}}:K_1] =\prod_{p|m}\frac{\varphi_K(p^{e_p})/\varphi(p^{e_p})}{\#\mu_K/2}.
\end{displaymath}
Using these degrees and $[K_1:K]=h_K$, we arrive at the desired conclusion
\begin{displaymath}
[K_m:L] = (\#\mu_K/2)^{\max(0,t-1)}
\end{displaymath}
where $t$ is the number of distinct primes dividing $m$.
\end{proof}

\begin{theorem} If the N\'eron--Severi rank of $\Abar$ is $4$, then we have an upper bound
\begin{displaymath}
\#\frac{\Br(X)}{\Br_1(X)}\leq 8.80\cdot 10^{191}C_3^2[k:\Q]^{16}(h(A)+\frac{\log C_3}2+\log [k:\Q]+25)^{12},
\end{displaymath}
where $C_3$ is the constant depending on $h(A)$ and $[k:\Q]$ that appeared in Theorem \ref{AtoEC}.
\end{theorem}
\begin{proof}
By Corollary \ref{rk4deg}, there exists a finite extension $k_2$ of $k$ of degree at most $C_5$ such that there exists an elliptic curve $E''$ over $k_2$ whose endomorphism ring is $\OO_K$ such that there is an isogeny $E''\times E''\rightarrow A_{k_2}$ of degree at most $C_4$, where $C_4$ and $C_5$ are constants depending on $h(A)$ and $[k:\Q]$. By Remark \ref{rmk_bc}, we have
\begin{displaymath}
\#\frac{\Br(X)}{\Br_1(X)}\leq \#\frac{\Br(X_{k_2})}{\Br_1(X_{k_2})}.
\end{displaymath}
By Lemma \ref{IsogBr}, we have
\begin{displaymath}
\#\frac{\Br(X_{k_2})}{\Br_1(X_{k_2})}\leq C_4^{6-\bar{r}} \#\frac{\Br(E''\times E'')}{\Br_1(E''\times E'')}=C_4^{2} \#\frac{\Br(E''\times E'')}{\Br_1(E''\times E'')}.
\end{displaymath}
By Theorem \ref{Newton}, we have
\begin{displaymath}
C_4^{2} \#\left(\frac{\Br(E''\times E'')}{\Br_1(E''\times E'')}\right)\leq 12.5^4C_4^2[k_2:\Q]^4=12.5^4C_2^2C_5^4[k:\Q]^4
\end{displaymath}
and we obtain the desired inequality by the formulas of $C_4$ and $C_5$ in Corollary \ref{rk4deg}.
\end{proof}

\subsubsection{Case (2) and Case (3)}
The main input for these cases is Theorems \ref{AtoEC}, \ref{effFalEC} combined with the following result of Skorobogatov and Zarhin.

\begin{proposition}[\cite{New}*{Lem.~2.1} and \cite{SZ12}*{Prop.~3.3}]\label{SZBr}
There is an isomorphism between abelian groups
\begin{displaymath}
\left(\frac{\Br(E\times E')}{\Br_1(E\times E')}\right)_{\ell^{\infty}}\cong\frac{\Br(E\times E')_{\ell_\infty}}{\Br_1(E\times E')_{\ell_\infty}}
\end{displaymath}
and there is an isomorphism between abelian groups
\begin{displaymath}
\frac{\Br(E\times E')_n}{\Br_1(E\times E')_n}\cong\frac{\Hom_\Gamma(E_n,E'_n)}{(\Hom(\Ebar,\Epbar)\otimes \Z/n\Z)^\Gamma}.
\end{displaymath}
\end{proposition}

\begin{cor}\label{BrEEp}
If either $E\neq E'$\footnote{Remember: by assumption this means that $E$ is not isogenous to $E'$.} or $\End_{\bar{k}}E=\Z$ hold, then 
\begin{displaymath}
\#\frac{\Br(E\times E')}{\Br_1(E\times E')}\leq C(h(E),h(E'), [k:\Q])=(C_2!)^{4},
\end{displaymath}
where $C_2$ is the constant depending on $h(E)$, $h(E')$ and $[k:\Q]$ that appeared in Theorem \ref{FalECnp}.
\end{cor}
\begin{proof}
By Theorem \ref{effFalEC}, for any prime $\ell>C_2$ and any positive integer $n$, we have that $\Hom_\Gamma(E_{\ell^n},E'_{\ell^n})/(\Hom(\Ebar,\Epbar)_{\ell^n})^\Gamma$ is trivial. Moreover, for $\ell\leq C_2$, let $\ell^{m(\ell)}$ be the largest $\ell$-power no greater than $C_2$. Theorem \ref{effFalEC} implies that, for any positive integer $m$, the group $\Hom_\Gamma(E_{\ell^m},E'_{\ell^m})/(\Hom(\Ebar,\Epbar)_{\ell^m})^\Gamma$ is killed by ${\ell^{m(\ell)}}$. Then by the second isomorphism in Proposition~\ref{SZBr}, $\tfrac{\Br(E\times E')_n}{\Br_1(E\times E')_n}$ is zero when $n$ is any power of a prime $\ell>C_2$ and is killed by $\ell^{m(\ell)}$ when $n$ is any power of a prime $\ell\leq C_2$. Then the first isomorphism in Proposition \ref{SZBr} implies that $\displaystyle \left(\tfrac{\Br(E\times E')}{\Br_1(E\times E')}\right)_{\ell^\infty}$ is zero when $\ell>C_2$ and is killed by ${\ell^{m(\ell)}}$ when $\ell\leq C_2$. Since $\bar{r}\geq 2$, by Theorem \ref{theorem: geometricBrauer} we have
\begin{displaymath}
\#\left(\frac{\Br(E\times E')}{\Br_1(E\times E')}\right)_{\ell^\infty}\leq \#\Br(\EEpbar)_{\ell^{m(\ell)}}\leq \ell^{4m(\ell)},
\end{displaymath}
for all $\ell \leq C_2$ and the order is $1$ for $\ell>C_2$. Hence we obtain the desired inequality by taking the product over all $\ell \leq C_2$.
\end{proof}

\begin{theorem}
If, after some finite field extension, $A$ is isogenous to $E\times E'$, where either $E$ is not isogenous to $E'$ or $\End_{\bar{k}}E=\Z$, then there is an effectively computable bound on 
$\displaystyle \#\frac{\Br(X)}{\Br_1(X)}$ depending only on $h(A)$ and $[k:\Q]$.
\end{theorem}
\begin{proof}
We may assume that $E=E'$ if $E$ and $E'$ are isogenous. Let $k_1$ be the smallest finite extension of $k$ such that $E$ and $E'$ are defined.
By Theorem \ref{AtoEC}, there is an isogeny $E\times E'\rightarrow A_{k_1}$ of degree at most $C_3$ and hence by Remark \ref{rmk_bc} and Lemma \ref{IsogBr}, we have
\begin{displaymath}
\#\frac{\Br(X)}{\Br_1(X)}\leq \#\frac{\Br(X_{k_1})}{\Br_1(X_{k_1})}\leq C_3^4\#\frac{\Br(E\times E')}{\Br_1(E\times E')}.
\end{displaymath}
Then by Corollary \ref{BrEEp}, we have
\begin{displaymath}
\#\frac{\Br(X)}{\Br_1(X)}\leq C_3(h(A),[k:\Q])^4(C_2(h(E),h(E'),[k_1:\Q])!)^4.
\end{displaymath}
Finally, in order to bound the right-hand side in terms of $h(A)$ and $[k:\Q]$, we use the inequalities $[k_1:\Q]\leq 4\cdot 18^8[k:\Q]$ and $\max(h(E),h(E'),1)\leq h(A)+\frac 32+\frac 12 \log C_3$.
\end{proof}

\section{Computations on rank 1}\label{computations}
In this section we discuss some computations in order to determine $\Br_1(X)/\Br_0(X)$ through $\HH^1(k,\NS(\ol{X}))$ using \textsc{Magma}, where the geometric N\'{e}ron--Severi rank of $X$ is 1. Recall that the N\'{e}ron--Severi lattice of a Kummer surface is determined by the sixteen 2-torsion points on the associated abelian surface and its N\'{e}ron--Severi lattice. A principally polarized abelian surface is the Jacobian of a genus 2 curve $C$ and its 2-torsion points correspond to the classes $p_i-p_j$ of differences of the six ramification points of $C\to\PP^1$.

First we need to fix some ordering. Let $\{p_1,\ldots,p_6\}$ be the ramification points of $C$. Then on $\Jac(C)[2]=\{0,p_i-p_j:i<j\}$ the following additive rule holds
\begin{displaymath}
p_i-p_j = p_k-p_l+p_n-p_m
\end{displaymath}
where $\{i,j\}$ and $\{k,l,m,n\}$ are two complementary subsets of $\{1,\ldots,6\}$.

\begin{lemma}
The set $\{p_1-p_2=:v_1, p_1-p_3=:v_2, p_1-p_4=:v_3, p_1-p_5=:v_4\}$ forms a basis of $\Jac(C)[2]\cong \F_2^4$.
\end{lemma}
\begin{proof}
In order to write 0 as a linear combination of these elements (over $\F_2$), we need to use an even number. Since any two of these are different, this may only be done using all four of them. However, the sum of these four elements is $p_2-p_3+p_4-p_5=p_1-p_6\neq 0$.
\end{proof}

We order the 2-torsion elements in terms of $p_i-p_j$ and in terms of $v_i$ in Table \ref{order_e}.

\begin{table}[h]
\begin{tabular}{|c|c|}
\hline
$e_1=0$ &  $e_9=p_1-p_5=v_4$\\
$e_2=p_1-p_2=v_1$ &  $e_{10}=p_2-p_5=v_1+v_4$\\
$e_3=p_1-p_3=v_2$ &  $e_{11}=p_3-p_5=v_2+v_4$\\
$e_4=p_2-p_3=v_1+v_2$ &  $e_{12}=p_4-p_6=v_1+v_2+v_4$\\
$e_5 = p_1-p_4=v_3$ & $e_{13}=p_4-p_5=v_3+v_4$\\
$e_6=p_2-p_4=v_1+v_3$ & $e_{14}=p_3-p_6=v_1+v_3+v_4$\\
$e_7=p_3-p_4=v_2+v_3$ & $e_{15}=p_2-p_6=v_2+v_2+v_4$\\
$e_8=p_5-p_6=v_1+v_2+v_3$ & $e_{16}=p_1-p_6=v_1+v_2+v_3+v_4$\\
\hline
\end{tabular}
\caption{}
\label{order_e}
\end{table}
The Galois action is defined by a subgroup of $S_6$, acting on the six ramification points $p_i$ and hence on the set of $e_i$. This action defines $S_6$ as a subgroup of $S_{16}$. We know that $S_6$ is generated by the two elements $(1,2)$ and $(1,2,3,4,5,6)$, so to determine the map $S_6\to S_{16}$ we need only specify the images of $(1,2)$ and $(1,2,3,4,5,6)$.

\begin{lemma}\label{action_S6}
Let $\rho\colon S_6\to S_{16}$ be the map that represents the action of $S_6$ on the sixteen 2-torsion points $e_i$. Then
\begin{displaymath}
\rho((1,2)) = (3,4)(5,6)(9,10)(15,16)
\end{displaymath}
and
\begin{displaymath}
\rho((1,2,3,4,5,6)) = (2,4,7,13,8,16)(3,6,11,12,9,15)(5,10,14)
\end{displaymath}
hold.
\end{lemma}
\begin{proof}
Direct computation on the elements in Table \ref{order_e}, e.g. $\rho((1,2))$ maps $e_3=p_1-p_3$ to $p_2-p_3=e_4$.
\end{proof}

Using the description from \cite{LP80}*{Prop.~3.4 and 3.5} as explained in section \ref{rank1}, the lattice $K$ is generated by $\bigoplus_{i=1}^{16}\Z\pi_*E_i$ together with lifts from polynomials in four variables with values in $\frac12\Z/\Z$ of degree at most 1. These are generated as an abelian group by $x_1,x_2,x_3,x_4,1$, where the set of $x_i$'s is dual to the set of $v_j$'s in the sense $x_i(v_j)=\delta_{ij}$. We identify the set of exceptional curves with the set of 2-torsion points in the natural way by identifying $E_i$ and $e_i$ for each $i=1,\ldots,16$.

From a theoretical perspective, one could use the approach as lain out in section \ref{rank1} in order to calculate $\NS(X)$, but for the case $\rk\NS(\ol A)=1$, it turns out that there is an easier approach which involves knowing the index of $\pi_*\NS(A)\oplus K$ in $\NS(X)$. 

\begin{lemma}\label{index_NS}
Let $A$ be an abelian surface of N\'{e}ron--Severi rank $\rho$, write $X=\Kum(A)$ and let $K$ be the saturation of $\bigoplus_{i=1}^{16}\Z \pi_*E_i$ inside $\NS(X)$. Then the index of $\pi_*\NS(A)\oplus K$ inside $\NS(X)$ is $2^{\rho}$.
\end{lemma}
\begin{proof}
Write $t=|\Disc \NS(A)|$, then also $t=|\Disc T(A)|$ holds, where $T(A)$ is the transcendental lattice of $A$, since $\HH^2(A,\Z)$ is unimodular. We have equality of ranks 
\begin{displaymath}
\rk T(X)=\rk T(A)=6-\rho,
\end{displaymath}
and hence $|\Disc T(X)|=t\cdot 2^{6-\rho}$ from which follows $|\Disc \NS(X)|=t\cdot 2^{6-\rho}$ since $\HH^2(X,\Z)$ is unimodular.

Let $L=\pi_*\NS(A)$. Then $\rk L=\rho$ and $|\Disc L|=2^{\rho}t$ hold.

We use the chain of inclusions
\begin{displaymath}
L\oplus K\subset \NS(X)\subset \NS(X)^\vee \subset L^\vee\oplus K^\vee
\end{displaymath}
The index of $L\oplus K\subset L^\vee\oplus K^\vee$ is $2^\rho t\cdot 2^6$ (see section \ref{rank1} for the discriminant of $K$) and combining with the discriminants above, we find the statement of the lemma.
\end{proof}

From now on, assume $\rho=1$, i.e. the geometric N\'{e}ron--Severi rank of $X$ is 17. Let $l$ be the push-forward of the theta-divisor on $A$. Then $l^2=4$ and by Lemma \ref{index_NS}, the index of $\Lambda:=\langle l\rangle\oplus K$ in $\NS(\ol X)$ is 2. It therefore suffices to find a single element $D\in \NS(\ol X)$ such that $2D$ is an element of $\Lambda$ but $D$ itself is  not. Then $\Lambda$ and $D$ together span $\NS(\ol X)$. 

\begin{lemma}\label{oneNSiso}
Up to isomorphism there is only one index 2 even overlattice of $\Lambda$. 
\end{lemma}
\begin{proof}
Even overlattices of index 2 correspond to isotropic subgroups of the discriminant group $D(\Lambda)=D(\pi_*\NS(\ol A))\oplus D(K)$ of order 2. Since $K$ is saturated, a generating element of such a subgroup projects to an element of $D(\pi_*\NS(\ol A))$ which has order exactly 2. Since $D(\pi_*\NS(\ol A))$ is isomorphic to $\frac14\Z/\Z$, there is only one such element, which has square $1\bmod 2$. We therefore need to consider order 2 elements of square $1\bmod 2$ in $D(K)$. Since we remember the intersection form on $D(K)$ from section \ref{rank1}, we easily see that there are four such elements, with coordinates $(1,0,0,0,0,1)$, $(0,1,0,0,1,0)$, $(0,0,1,1,0,0)$ and $(1,1,1,1,1,1)$. By calculating the centralizer of the intersection matrix of $D(K)$ inside $\GL_6(\F_2)$, that is $\calO(D(K))$, it is easily found that each of these lie in the same orbit under the action of $\calO(D(K))$. 
\end{proof}

It is worthwhile to remark that the Galois action on the 2-torsion points of $A$ induces an action on $D(K)$ and only one of the four elements in the previous proof is invariant under the action of the full symmetric group $S_6$, which in our chosen basis is $(1,1,1,1,1,1)$.

\begin{lemma}
The element $D=\frac12(\pi_*E_1+\pi_*E_8+\pi_*E_{12}+\pi_*E_{14}+\pi_*E_{15}+\pi_*E_{16}+l)$ together with $\Lambda$ spans $\NS(\ol X)$. 
\end{lemma}
\begin{proof} We already know that the coefficient of $l$ is non-zero since $K$ is saturated in $\NS(\ol X)$, and by adding a suitable element of $2\Lambda$ to $D$, we can write $D=\frac{1}{2}l+\frac{1}{2}\sum_{i=1}^{16}a_i\pi_*E_i$, where for each $i$ we take $a_i\in \{0,\frac12,1,\frac32\}$. 

By intersecting $D$ with any of the $\pi_*E_i$, we find $a_i\in\{0,1\}$ since the intersection needs to be integral. From $D^2\in 2\Z$ we deduce $\sum_{i=1}^{16}a_i \equiv 2 \bmod 4$.
Furthermore, the projection of $D$ to $D(K)$ needs to be one of the four elements from the proof of Lemma \ref{oneNSiso}. In order to ensure that the lattice we generate is a Galois module for any subgroup of $S_6$, the element $D$ from the statement is chosen so that it projects to the unique $S_6$-invariant one. 
\end{proof}

Now that we have computed $\NS(\ol X)$, we can have \textsc{Magma} take Galois cohomology by applying the action from Lemma \ref{action_S6} and we find \begin{displaymath}
\HH^1(k,\NS(\ol X))=1.
\end{displaymath}

We can furthermore consider the case where the Galois group is not the full $S_6$. The \textsc{Magma} computations also yield the following:
\begin{proposition}\label{exceptions}
Up to conjugation there are only three subgroups $H$ of $S_6$ for which $\HH^1(H,\NS(\ol X))$ is non-trivial: one of order 4 (isomorphic to $\Z/2\Z \times \Z/2\Z$), one of order 12 (isomorphic to $A_4$) and one of order 60 (isomorphic to $A_5$). In each of these cases we find $\HH^1(H,\NS(\ol X))\cong \Z/2\Z$.
\end{proposition}

\section{An example}\label{section:example}

In this section we compute a concrete bound as stated in Theorem \ref{effFal}. Let us consider the genus $2$ curve defined over $\Q$ by:
\begin{displaymath}
C:\, y^2 = x^6+x^3+x+1.
\end{displaymath}

Let $A$ denote the Jacobian of $C$. Thanks to the algorithm provided by Elsenhans and Jahnel in \cite{EJ} we compute the N\'{e}ron--Severi rank of $A$ and we obtain that its geometric N\'{e}ron--Severi rank is $1$. By Theorem \ref{32} we know $\End(A)=\Z$.

Since $x^6+x^3+x+1=(x+1)(x^2+1)(x^3-x^2+1)$, the splitting field $F$ of $x^6+x^3+x+1$ is the composite field of $\Q(\sqrt{-1})$ and the splitting field $F_1$ of $x^3-x^2+1$. The Galois group $\Gal(F/\Q)$ has $12$ elements and two normal subgroups: $\Z/2\Z$ and $S_3$. By Proposition \ref{exceptions}, the only exceptional subgroup with $12$ elements is $A_4$. Since the only nontrivial normal subgroup of $A_4$ has $4$ elements, $\Gal(F/\Q)$ cannot be one of the exceptional subgroups of $S_6$. Therefore the algebraic Brauer group is trivial.

To compute the bound of Theorem \ref{effFal} we need to compute the Faltings height of the abelian surface $A$. By \cite{P14}*{Thm.~2.4} and noticing that we use the same bound for all places of $\Q(A[12])$ above a given rational place, we have 
\begin{displaymath}
h(A)\leq -\log(2\pi^2)+ \tfrac{1}{10}\log \left(2^{-12}\Disc_{6}\left(4(x^6+x^3+x+1)\right)\right)-\log\left(2^{-1/5}|J_{10}|^{1/10}\det(\Im \tau)^{1/2}\right),
\end{displaymath}
with $2^{-12}\Disc_{6}\left(4(x^6+x^3+x+1)\right)=2^{12}\cdot25\cdot23$, $|J_{10}|=0.001921635$ and 
\begin{displaymath}
\tau=
\begin{pmatrix}
-1.49097+ 1.64505i&-0.50000+ 0.98058i\\
-0.50000 + 0.98058i&
    -1.50903 + 1.64505i
    \end{pmatrix}.
\end{displaymath}
Hence $h(A)\leq -0.79581.$    
In our situation we have $k=\Q$ and $h(A)\leq -0.79581$, so we can bound $M$ by plugging these into
\begin{displaymath} 
M\leq 2^{4664}c_1^{16}c_2(k)^{256}\left(2h(A)+\tfrac{8}{17}\log[k:\Q]+8\log c_1+128 \log c_2(k) +1503\right)^{512}
\end{displaymath}
with $c_1=4^{11}\cdot 9^{12}$ and $c_2(k)=7.5\cdot 10^{47}[k:\Q]$.

Using \textsc{Magma} we get 
\begin{displaymath}
M \leq C=8.7\times 10^{16100}.
\end{displaymath}

Let $X = \Kum(A)$.
Following the proofs in Subsection~\ref{subsec: general case}, we have
\begin{displaymath}
\Br(\overline{X})^\Gamma <4\prod_{\ell<C} C^{50}<4C^{50C}<4\cdot 10^{7.5\cdot 10^{16106}}. 
\end{displaymath}
Hence we conclude that
\begin{displaymath}
|\Br(X)/\Br_0(X)| < 4\cdot 10^{7.5\cdot 10^{16106}}.
\end{displaymath}


\end{document}